\documentclass{proc-l}

\issueinfo{3}{4}{December}{2009}

\copyrightinfo{2009}{Aulona Press  \emph{(Albanian J. Math.)}}

\def\issn{{\sc ISSN} 1930-1235: }
\def\issueyear{2009}

\PII{\issn  (\issueyear )}

\pagespan{131}{160}

\usepackage{amssymb}
\usepackage{amsmath,amssymb,amsbsy,amsfonts,amsthm,latexsym,amsopn,amstext, amsxtra,euscript,amscd, hyperref}
\usepackage{graphicx}
\usepackage{longtable}
\usepackage{array}

\usepackage{xy}
\input xy \xyoption{all}

\newtheorem{thm}{Theorem}[section]

\newtheorem{lemma}{Lemma}

\newtheorem{remark}{Remark}
\theoremstyle{definition}

\newcommand{\Aut}{\mbox {Aut}}
\newcommand{\ch}{\mbox {char }}

\newcommand{\bP}{{\mathbb P}^1}

\newcommand\F{\mathbb F}

\newcommand\C{\bold C}

\newcommand{\cH}{\mathcal H}

\newcommand\X{\mathcal X}
\newcommand\G{\bar G}
\def\si{\sigma}
\def\ta{\tau}
\def\m{\mu}
\def\p{\phi}
\def\P{\Phi}
\def\d{\delta}
\def\U{\mathcal U}

\def\s{s}
\def\t{t}
\def\r{r}

\begin{document}

\title[Automorphisms of curves]{Automorphism groups of cyclic curves defined over finite fields of any characteristics}




\keywords{algebraic curves;  automorphism groups}

\maketitle

\begin{center}
{\sc R. Sanjeewa}\\
\vspace*{1ex}
\emph{Dept. of Mathematics\\ Oakland University\\Rochester, MI. USA.\\
rsanjeew@oakland.edu}
\end{center}

\vspace*{3ex}

\begin{abstract}
In this paper  we determine automorphism groups of cyclic algebraic curves defined over finite fields
of any characteristic.
\end{abstract}

\section{Introduction}

Let $g\geq 2$ be a fixed integer and $k$ an algebraically closed field of characteristic $p\geq 0$ and let $\X_g$ is an irreducible algebraic curve defined over $k$. Determining the list of group which occur as automorphism groups of $\X_g$ is very old problem in mathematics. Mathematicians are working on this problem more than a century now. In 1893 Hurwitz proved that the order of such group is $\leq 84 (g-1)$ when $p=0$; see \cite{AH}. After 80 years, Stichtenoth et al. proved that the  bound is $16 g^4$ for $p> 0$; see \cite{ST2}.  In any case, the group of automorphisms is a finite group.  There are hundreds of papers on the structure of such groups, determining the equation of the curve when the group is given, determining the group when the curve is given, etc. T. Shaska determines the list of groups for hyperelliptic curves when $p=0$, see \cite{TS} and K. Magaard et. al. determine the list of groups for any given $g\geq 2$ when $p=0$. Such results are based on an exhaustive computer search of all possible ramification structures for a given $g$ and a deep understanding of Hurwitz spaces for a given genus $g$, a group $G$, and the ramification structure for the covering $\X_g \to \X^G_g$. The case of positive characteristic is still an open problem.

In most cases there is a cyclic subgroup $C_n \vartriangleleft G$ such that $g(\X^{C_n})=0$. Such curves are called \emph{cyclic curves}. In this paper we determine groups $G$ which occur as automorphism groups of cyclic curves in any characteristic and for any genus $g\geq 2$.

In section 2, we cover basic facts on automorphism groups of cyclic curves. Let $G=\Aut(\X_g)$ automorphism group of cyclic curve $\X_g$, $C_n=\left\langle w\right\rangle$ such that $g(\X^{C_n})=0$. The group $\G:= \Aut(\X_g)/\langle w \rangle$ is called the reduced automorphism group. This group $\G$ is embedded in $PGL_2(k)$ and therefore is isomorphic to one of $C_m $, $ D_m $, $ A_4 $, $ S_4 $, $ A_5 $, \emph{a semidirect product of elementary Abelian group  with  cyclic group}, $ PSL(2,q)$ and $PGL(2,q)$ cf. Lemma ~\ref{l1}. We determine a rational functions $\phi(x)$ that generates the fixed field $k(x)^{\G}$.

In the section 3, we determine ramification signature of each cover $\P(x):\X \to \X^G$ by using the ramification of $\G$. By considering the lifting of ramified points of the cover $\P$, we divide each $\G$ into sub cases. Then we are able to find automorphism group $G$ for each sub cases. For some cases, we give presentation for $G$. The moduli space of covers $\P$ with fixed group $G$ and ramification signature $\C$ is a Hurwitz space $\cH$. There is a map from $\cH$ to the moduli space of genus $g$ algebraic curves $\cH_g$. The image of this map is a subvariety of $\cH_g$ and denoted by $\cH(G,\C)$. Since we know the signature of the curve, we use Hurwitz genus formula to calculate dimension of $\cH(G,\C)$. We list all possible signatures $\C$ and dimension of the locus  $\cH(G,\C)$. Then we list automorphism groups as theorems for each $\G$.

In section 4, we combine Theorems ~\ref{th9} - ~\ref{th3} altogether to make main theorem. In our main theorem, we list all possible automorphism groups of genus $g\geq 2$ cyclic curves define over the finite field of characteristic $p\neq2$. We are able to give presentations for some of automorphism groups.

\medskip

\textbf{Notation :} Through this paper $k$ denotes an algebraically closed field of characteristic $\neq $2,
$g$ an integer $\geq $2, and $\X_g$ a cyclic curve of genus $g$ defined over $k$. For a given
curve $\X$, $g(\X)$ denotes its genus.


\section{Preliminaries}

Let $k$ be an algebraically closed field of characteristic $p$ and $\X_g$ be a genus $g$ cyclic
curve given by the equation $y^n=f(x)$ for some $f \in k[x]$. Let $K:=k(x,y)$ be the function field of $\X_g$. Then
$k(x)$ is degree n genus zero subfield of $K$. Let $G=\Aut(K/k)$. Since $C_n:=Gal(K/k(x))=\langle w
\rangle$, with $w^n=1$ such that $\langle w \rangle \lhd G$, then group $\G:=G/C_n$ and $\G \leq
PGL_2(k)$. Hence $\G$ is isomorphic to one of the following: $C_m $, $ D_m $, $ A_4 $, $ S_4 $, $ A_5
$, \emph{semidirect product  of elementary  Abelian  group  with  cyclic  group}, $ PSL(2,q)$ and $PGL(2,q)$,
see \cite{VM}.

The group $\G$ acts on $k(x)$ via the natural way. The fixed field is a genus 0 field, say $k(z)$. Thus $z$
is a degree $|\G|$ rational function in $x$, say $z=\p(x)$. We illustrate with the following diagram:
\[
\xymatrix{
K =k(x, y)\ar@{-}[d]^{\, \, \, C_n} \ar@/_1.8pc/[dd]_{\, \, G} \\
k(x, y^n) \ar@{-}[d]^{\, \, \, \G}  \\ k (z)  \\ }
\qquad \qquad \qquad \xymatrix{ \X_g \ar[d]_{\, \, \phi_0 }^{\, \, \,
C_n} \ar@/_1.8pc/[dd]_{\, \, \Phi} \\
\bP \ar[d]_{\, \, \, \phi}^{\, \, \, \G} \\ \bP  \\ }
\]

Let $\p_0:\X_g \to \bP$ be the cover which corresponds to the degree n extension $K/k(x)$. Then $\P:=\p \circ \p_0$ has monodromy group $G:=\Aut(\X_g)$. From the basic covering theory, the group $G$ is embedded in the group $S_l$ where $l=deg$ $\P$. There is an $r$-tuple $\overline{\si}:=(\si_1,...,\si_r)$, where $\si_i \in S_l$ such that $\si_1,...,\si_r$ generate $G$ and $\si_1...\si_r=1$. The signature of $\P$ is an $r$-tuple of conjugacy classes $\C:=(C_1,...,C_r)$ in $S_l$ such that $C_i$ is the conjugacy class of $\si_i$. We use the notation $n$ to denote the conjugacy class of permutations which is cycle of length $n$. Using the signature of $\p:\bP \to \bP$ one finds out the signature of $\P:\X_g \to \bP$ for any given $g$ and $G$.

Let $E$ be the fixed field of G, the Hurwitz genus formula states that
\begin{equation}\label{e1}
2(g_K-1)=2(g_E-1)|G|+deg(\mathfrak{D}_{K/E})
\end{equation}
with $g_K$ and $g_E$ the genera of $K$ and $E$ respectively and $\mathfrak{D}_{K/E}$ the different of $K/E$. Let $\overline{P}_1, \overline{P}_2, ..., \overline{P}_r$ be ramified primes of $E$. If we set $d_i=deg(\overline{P}_i)$ and let $e_i$ be the ramification index of the $\overline{P}_i$ and let $\beta_i$ be the exponent of $\overline{P}_i$ in $\mathfrak{D}_{K/E}$. Hence, (1) may be written as
\begin{equation}\label{e2}
2(g_K-1)=2(g_E-1)|G|+|G|\sum_{i=1}^{r}\frac{\beta_i}{e_i}d_i
\end{equation}

If $\overline{P}_i$ is tamely ramified then $\beta_i=e_i-1$ or if $\overline{P}_i$ is wildly ramified then $\beta_i=e_i^*q_i+q_i-2$ with $e_i=e_i^*q_i$, $e_i^*$ relatively prime to $p$, $q_i$ a power of $p$ and $e_i^*|q_i-1$. For fixed $G$, $\C$ the family of covers $\P:\X_g\to \bP$ is a Hurwitz space $\cH(G,\C)$. $\cH(G,\C)$ is an irreducible algebraic variety of dimension $\d(G,\C)$. Using equation ~(\ref{e2}) and signature $\C$ one can find out the dimension for each $G$.

Next we want to determine the cover $z=\p(x):\bP \to \bP$ for all characteristics. Notice that the case of $char(k)=0$ is known, see \cite{TS}.

We define a semidirect product of elementary Abelian group  with  cyclic group as follows.
\[K_m:=\left\langle\left\{ \si_a, \t |a \in \U_m\right\}\right\rangle \]
where $\t(x)=\xi^2x, \quad \si_a(x)=x+a,$ for each $a \in \U_m$,
\[\U_m :=\{a \in k| (a\prod_{j=0}^{\frac{p^t-1}{m}-1}(a^m-b_j))=0\}\]
$b_j \in \F^*_q$, $m|p^t-1$ and $\xi$ is a primitive $2m$-th root of unity. Obviously $\U_m$ is a subgroup of the additive group of $k$.

\begin{lemma}\label{l1}
Let $k$ be an algebraically closed field of characteristic $p$, $\G$ be a finite subgroup of $PGL_2(k)$ acting on the field $k(x)$. Then, $\G$ is isomorphic to one of the following groups $C_m$, $D_m$, $A_4$, $S_4$, $A_5$, $U= C_p^t$, $K_m$, $PSL_2(q)$ and $PGL_2(q)$, where $q=p^f$ and $(m,p)=1$. Moreover, the fixed subfield $k(x)^{\G}=k(z)$ is given by Table ~\ref{t1}, where $\alpha =\frac{q(q-1)}{2}, \quad \beta= \frac{q+1}{2}$. $H_t$ is a subgroup of the additive group of $k$ with $| H_t | = p^t$ and $b_j \in k^*$.
\end{lemma}

\begin{table}
\begin{center}
\begin{tabular}{cccc}
$Case$ & $\G$ & $z$ & $Ramification$  \\
\hline \\
1 & $C_m$, $(m,p)=1$& $x^m$ & $(m,m)$\\ \\
2 & $D_{2m}$, $(m,p)=1$& $x^m+\frac{1}{x^m}$ & $(2,2,m)$\\ \\
3 & $A_4, \, p\neq 2, 3$ & $\frac{x^{12}-33x^8-33x^4+1}{x^2(x^4-1)^2}$ & $(2,3,3)$\\ \\
4 & $S_4, \, p\neq 2, 3$ & $\frac{(x^8+14x^4+1)^3}{108(x(x^4-1))^4}$ & $(2,3,4)$\\ \\
5 & $A_5, \,p\neq 2, 3, 5$ & $\frac{(-x^{20}+228x^{15}-494x^{10}-228x^5-1)^3}{(x(x^{10}+11x^5-1))^5}$ & $(2,3,5)$\\ \\
  & $A_5, \,p=3$ & $\frac{(x^{10}-1)^6}{(x(x^{10}+2ix^5+1))^5}$ & $(6,5)$\\ \\
6 & $U$ & $  \displaystyle{\prod_{a \in H_t}} (x+a)$ & $(p^t)$\\ \\
7 & $K_m$ & $(x   \displaystyle{\prod_{j=0}^{\frac{p^t-1}{m}-1} } (x^m-b_j))^m$ & $(mp^t,m)$ \\ \\
8 & $PSL(2,q), \,p\neq 2$ & $\frac{((x^q-x)^{q-1}+1)^{\frac{q+1}{2}}}{(x^q-x)^{\frac{q(q-1)}{2}}}$ & $(\alpha,\beta)$ \\ \\
9 & $PGL(2,q)$ & $\frac{((x^q-x)^{q-1}+1)^{q+1}}{(x^q-x)^{q(q-1)}}$ & $(2\alpha,2\beta)$ \\ \\
\end{tabular}
\caption{Rational functions correspond to each $\G$} \label{t1}
\end{center}
\end{table}

\begin{proof}  It is well known that $\G$ is isomorphic to $C_n, D_n, A_4, S_4, A_5, U, K_m,$ $PSL(2,q)$ and $PGL(2,q)$; see
\cite{VM}.  Cases 1) .. 5) are the same as in characteristic zero; see \cite{ShTh}. We briefly display the
generators of $\G$ is such cases. The reader can check that $z$ is fixed by such generators.

\textbf{Case 1:} If $\G \cong C_m$ then $C_m =\left\langle \si \right\rangle$, where $\si(x)=\zeta
x$, $ \zeta$ is a primitive $m^{th}$ root of unity. So $\si(z)=(\zeta x)^m=\zeta^mx^m=x^m=z.$

\textbf{Case 2:} If $\G \cong D_{2m}$ then $D_{2m}=\left\langle \si, \t \right\rangle$, where
$\si(x)=\xi x$, $\t(x)=\frac{1}{x}$,  $\xi$ is primitive $(2m)^{th}$ root of unity. Hence, $\si$ and
$\t$ fix $z$.

\textbf{Case 3:} If $\G \cong A_4$ then $A_4=\left\langle \si, \m \right\rangle$, where
$\si(x)=-x$, $\m(x)=i\frac{x+1}{x-1}$, $i^2=-1$. Therefore, $\si(z)=z$ and $\m (z)=z$.

\textbf{Case 4:} If $\G \cong S_4$ then $S_4=\left\langle \si, \m \right\rangle$, where $\si(x)=ix$,
$\m(x)=i\frac{x+1}{x-1}$, $i^2=-1$. Therefore, $\si, \m$ fix $z$.

\textbf{Case 5:} If $\G \cong A_5$ then $A_5=\left\langle \si, \rho \right\rangle$, where
$\si(x)=\xi x$, $\rho(x)=-\frac{x+b}{bx-1}$, $\xi$ is primitive fifth root of unity and $b=-i(\xi+\xi^4)$,
$i^2=-1$. One can check that $\si, \rho$ fix $z$.

\textbf{Case 6:} If $\G \cong U$ then $U=\left\langle\left\{ \si_a|a\in H_t\right\} \right\rangle$, where
$\si_a(x)=x+a$ with $a \in H_t$. Therefore,
\begin{equation*}
\begin{split}
\si_a(z) & =\prod_{a_1 \in H_t}(x+a_1+a)=\prod_{a_2 \in H_t}(x+a_2)=\prod_{a \in H_t}(x+a)=z.\\
\end{split}
\end{equation*}
and $a_2 =a+a_1\in H_t$.

\textbf{Case 7:} If $\G \cong K_m$ then $K_m=\left\langle\left\{ \si_a, \t |a \in
\U_m\right\}\right\rangle$, where $\t(x)=\xi^2x$, $\si_a(x)=x+a$ for each $a\in
\U_m:=\{a \in k| (a\prod_{j=0}^{\frac{p^t-1}{m}-1}(a^m-b_j))=0\}\leq H_t$, $\xi$ is a primitive
$2m$-th root of unity. So,
\begin{equation*}
\begin{split}
\t(z)&=((\xi^2x)\prod_{j=0}^{\frac{p^t-1}{m}-1}((\xi^2x)^m-b_j))^m=(x\prod_{j=0}^{\frac{p^t-1}{m}-1}(x^m-b_j))^m=z.\\
\si_a(z) &
=((x+a)\prod_{j=0}^{\frac{p^t-1}{m}-1}((x+a)^m-b_j))^m=(x\prod_{j=0}^{\frac{p^t-1}{m}-1}(x^m-b_j))^m=z.
\end{split}
\end{equation*}

\textbf{Case 8:} If $\G \cong PSL( 2,q)$ then $PSL(2,q)=\left\langle \si, \t, \p \right\rangle$, where
$\si(x)=\xi^2 x$, $\t(x)=-\frac{1}{x}$, $\p(x)=x+1$ and $\xi$ is a primitive $(q-1)$-th root of unity.
So,
\begin{equation*}
\begin{split}
\si(z) & =\frac{(((\xi^2 x)^q-(\xi^2 x))^{q-1}+1)^{\frac{q+1}{2}}}{((\xi^2 x)^q-(\xi^2
x))^{\frac{q(q-1)}{2}}}
=\frac{((x^q-x)^{q-1}+1)^{\frac{q+1}{2}}}{(x^q-x)^{\frac{q(q-1)}{2}}}=z.\\
\t(z) & =\frac{(((-\frac{1}{x})^q-(-\frac{1}{x}))^{q-1}+1)^{\frac{q+1}{2}}}{((-\frac{1}{x})^q-(-\frac{1}{x}))^{\frac{q(q-1)}{2}}}=\frac{((x^q-x)^{q-1}+1)^{\frac{q+1}{2}}}{(x^q-x)^{\frac{q(q-1)}{2}}}=z.\\
\p(z)
&=\frac{(((x+1)^q-(x+1))^{q-1}+1)^{\frac{q+1}{2}}}{((x+1)^q-(x+1))^{\frac{q(q-1)}{2}}}=\frac{((x^q-x)^{q-1}+1)^{\frac{q+1}{2}}}{(x^q-x)^{\frac{q(q-1)}{2}}}=z.
\end{split}
\end{equation*}

\textbf{Case 9:} If $\G \cong PGL(2,q)$ then $PGL(2,q)=\left\langle \si, \t, \p \right\rangle$, where
$\si(x)=\xi x$, $\t(x)=\frac{1}{x}$, $\p(x)=x+1$ and $\xi$ is a primitive $(q-1)$-th root of unity.
Hence,
\begin{equation*}
\begin{split}
\si(z) & =\frac{(((\xi x)^q-(\xi x))^{q-1}+1)^{q+1}}{((\xi x)^q-(\xi x))^{q(q-1)}}
=\frac{((x^q-x)^{q-1}+1)^{q+1}}{(x^q-x)^{q(q-1)}}=z.\\
\t(z) &= \frac{((( \frac {1} {x} )^q -( \frac{1} {x}))^{q-1}+1)^{q+1}  }
{(( \frac{1}{x})^q - (\frac{1} {x}))^{q(q-1)} } =\frac{((x^q-x)^{q-1}+1)^{q+1} } {(x^q-x)^{q(q-1)} }=z .\\
\p(z) & =\frac{(((x+1)^q-(x+1))^{q-1}+1)^{q+1} } {((x+1)^q-(x+1))^{q(q-1)} } =
\frac{((x^q-x)^{q-1}+1)^{q+1} } {(x^q-x)^{q(q-1)} } =z
\end{split}
\end{equation*}

This completes the proof.
\end{proof}


\section{Automorphism groups of a cyclic curves}

In this section we determine groups which occur as automorphism group $G$ of genus $g\geq 2$ cyclic curves, their signatures and the dimension of the locus $\cH(G,\C)$. We know that $\G:=G/G_0$, where $G_0:=Gal(k(x,y)/k(x))$ and  $\G$ is isomorphic to $C_m$, $D_m$, $A_4$, $S_4$, $A_5$, $U$, $K_m$, $PSL(2,q)$, $PGL(2,q)$. By considering the lifting of ramified points in each $\G$, we divide each $\G$ into sub cases. We determine signature of each sub case by looking the behavior of lifting and ramification of $\G$. Using that signature and Equation ~\ref{e2} we calculate $\delta$ for each case. We list all possible automorphism groups $G$ as separate theorems for each $\G$.

\subsection{The case 2g+1$\geq$ p$>$ 5.} Throughout this subsection we assume that $2g+1\geq p> 5$.

\begin{remark}
The case $p>2g+1$ is same as $char=0$; see \cite{TS}

\end{remark}

\begin{thm} \label{th1}
Let $g \geq$ 2 be a fixed integer, $\X$ a genus $g$ cyclic curve, $G=\Aut(\X)$ and $C_n\triangleleft G$ such that $g(\X^{C_n})=0$. The signature of cover $\P(x):\X \to \X^G$ and dimension $\delta$ is given in Table ~\ref{t2}. In Table ~\ref{t2}, $m=|PSL_2(q)|$ for cases 38-41 and $m=|PGL_2(q)|$ for cases 42-45.
\end{thm}

\renewcommand{\arraystretch}{2}
\begin{longtable}{|c|c|c|c|}
\hline \hline
\multicolumn{1}{|c|}{$\#$} & \multicolumn{1}{|c|}{$\G$} & \multicolumn{1}{|c|}{$\delta(G,\C)$} & \multicolumn{1}{|c|}{$\C=(C_1,...,C_r)$}\\
\hline \hline
\endfirsthead
\hline \hline
\multicolumn{1}{|c|}{$\#$} & \multicolumn{1}{|c|}{$\G$} & \multicolumn{1}{|c|}{$\delta(G,\C)$} &  \multicolumn{1}{|c|}{$\C=(C_1,...,C_r)$}\\
\hline \hline
\endhead
\hline
\multicolumn{4}{r}{\itshape continued on the next page}\\
\endfoot
\multicolumn{2}{r}{ }
\endlastfoot
$1$ &$(p,m)=1$ & $\frac{2(g+n-1)}{m(n-1)}-1$ &  $(m,m,n,...,n)$ \\
$2$ & $C_m$ & $\frac{2g+n-1}{m(n-1)}-1$ &  $(m,mn,n,...,n)$\\
$3$ &  & $\frac{2g}{m(n-1)}-1$ &   $(mn,mn,n,...,n)$\\
\hline
$4$ & $(p,m)=1$ & $\frac{g+n-1}{m(n-1)}$ &  $(2,2,m,n,...,n)$  \\
$5$ &  & $\frac{2g+m+2n-nm-2}{2m(n-1)}$ &  $(2n,2,m,n,...,n)$  \\
$6$ & $D_{2m}$ & $\frac{g}{m(n-1)}$ &  $(2,2,mn,n,...,n)$ \\
$7$ & &$\frac{g+m+n-mn-1}{m(n-1)}$ &  $(2n,2n,m,n,...,n)$  \\
$8$ &  & $\frac{2g+m-mn}{2m(n-1)}$ &  $(2n,2,mn,n,...,n)$  \\
$9$ &  & $\frac{g+m-mn}{m(n-1)}$ &  $(2n,2n,mn,n,...,n)$  \\
\hline \hline
$10$ &  & $\frac{n+g-1}{6(n-1)}$ &  $(2,3,3,n,...,n)$  \\
$11$ &$ A_4$ & $\frac{g-n+1}{6(n-1)}$ &  $(2,3n,3,n,...,n)$ \\
$12$ &  & $\frac{g-3n+3}{6(n-1)}$ &  $(2,3n,3n,n,...,n)$  \\
$13$ &  & $\frac{g-2n+2}{6(n-1)}$ &  $(2n,3,3,n,...,n)$  \\
$14$ &  & $\frac{g-4n+4}{6(n-1)}$ &  $(2n,3n,3,n,...,n)$  \\
$15$ &  & $\frac{g-6n+6}{6(n-1)}$ &  $(2n,3n,3n,n,...,n)$  \\
\hline \hline
$16$ &  & $\frac{g+n-1}{12(n-1)}$ &  $(2,3,4,n,...,n)$  \\
$17$ &  & $\frac{g-3n+3}{12(n-1)}$ &  $(2,3n,4,n,...,n)$  \\
$18$ &  &  $\frac{g-2n+2}{12(n-1)}$ &  $(2,3,4n,n,...,n)$ \\
$19$ &  & $\frac{g-6n+6}{12(n-1)}$ &  $(2,3n,4n,n,...,n)$  \\
$20$ &  $S_4$ & $\frac{g-5n+5}{12(n-1)}$ &  $(2n,3,4,n,...,n)$  \\
$21$ &  & $\frac{g-9n+9}{12(n-1)}$ &  $(2n,3n,4,n,...,n)$  \\
$22$ &  & $\frac{g-8n+8}{12(n-1)}$ &  $(2n,3,4n,n,...,n)$  \\
$23$ &  & $\frac{g-12n+12}{12(n-1)}$  & $(2n,3n,4n,n,...,n)$  \\
\hline \hline
$24$ &  & $\frac{g+n-1}{30(n-1)}$ &  $(2,3,5,n,...,n)$  \\
$25$ &  & $\frac{g-5n+5}{30(n-1)}$ &  $(2,3,5n,n,...,n)$ \\
$26$ &  & $\frac{g-15n+15}{30(n-1)}$  & $(2,3n,5n,n,...,n)$  \\
$27$ &  & $\frac{g-9n+9}{30(n-1)}$ &  $(2,3n,5,n,...,n)$  \\
$28$ &  $A_5$ & $\frac{g-14n+14}{30(n-1)}$ &  $(2n,3,5,n,...,n)$  \\
$29$ &  & $\frac{g-20n+20}{30(n-1)}$ &  $(2n,3,5n,n,...,n)$  \\
$30$ &  & $\frac{g-24n+24}{30(n-1)}$ &  $(2n,3n,5,n,...,n)$  \\
$31$ &  & $\frac{g-30n+30}{30(n-1)}$ &  $(2n,3n,5n,n,...,n)$  \\
\hline \hline
$32$ &  & $\frac{2g+2n-2}{p^t(n-1)}-2$ &  $(p^t,n,...,n)$  \\
$33$ & $U$ & $\frac{2g+np^{t}-p^t}{p^t(n-1)}-2$ &   $(np^t,n,...,n)$ \\
\hline \hline
$34$ &  & $\frac{2(g+n-1)}{mp^t(n-1)}-1$ &   $(mp^t,m,n,...,n)$  \\
$35$ &  & $\frac{2g+2n+p^t-np^t-2}{mp^t(n-1)}-1$ &  $(mp^t,nm,n,...,n)$  \\
$36$ & $K_m$ & $\frac{2g+np^t-p^{t}}{mp^t(n-1)}-1$ &   $(nmp^t,m,n,...,n)$\\
$37$ &  & $\frac{2g}{mp^t(n-1)}-1$ &   $(nmp^t,nm,n,...,n)$ \\
\hline \hline
$38$& & $\frac{2(g+n-1)}{m(n-1)}-1$ &   $(\alpha,\beta,n,...,n)$  \\
$39$& $PSL_2(q)$  & $\frac{2g+q(q-1)-n(q+1)(q-2)-2}{m(n-1)}-1$ &   $(\alpha,n\beta,n,...,n)$  \\
$40$& & $\frac{2g+nq(q-1)+q-q^2}{m(n-1)}-1$ &   $(n\alpha,\beta,n,...,n)$  \\
$41$& & $\frac{2g}{m(n-1)}-1$ &   $(n\alpha,n\beta,n,...,n)$  \\
\hline \hline
$42$& & $\frac{2(g+n-1)}{m(n-1)}-1$ &   $(2\alpha,2\beta,n,...,n)$  \\
$43$& $PGL_2(q)$  & $\frac{2g+q(q-1)-n(q+1)(q-2)-2}{m(n-1)}-1$ &   $(2\alpha,2n\beta,n,...,n)$  \\
$44$& & $\frac{2g+nq(q-1)+q-q^2}{m(n-1)}-1$ &   $(2n\alpha,2\beta,n,...,n)$ \\
$45$& & $\frac{2g}{m(n-1)}-1$ &   $(2n\alpha,2n\beta,n,...,n)$ \\
\hline \hline
\caption{The signature $\C$ and dimension $\d$ for $\ch >5 $} \label{t2}
\end{longtable}

\begin{proof}
Let $n$ be the number of branch points of $\P$. Then $\d=n-3$; see \cite{MS}. We know that
$\p_0:\X_g\to \bP$ corresponds to degree $n$ extension $K/k(x)$.

\textbf{Case $\G \cong C_m$ :} The ramification of $\p:\bP_x \to \bP_z$ is $(m,m)$. i.e. $C_m=\left\langle \si, \ta | \si^m=\ta^m =\si\ta=1\right\rangle$. where $\ta= \si^{-1}$.

\textbf{(1)} If $\si$ and $\ta$ both lift to elements of order $m$ in $G$, then conjugacy classes
$\C=(m,m,n,...,n)$. By Riemann Hurwitz formula, we have
\[  2(g-1)=2(0-1)mn+mn\left(\left(\frac{m-1}{m}\right)\cdot 2 + \left(\frac{n-1}{n}\right)\left(\delta+1\right)\right)
\]
Then $\d=\frac{2(g+n-1)}{m(n-1)}-1$.\\

\textbf{(2)} If either $\si$ or $\ta$ lifts to an  element of order $mn$ in $G$, then ramification $\C=(mn,m,n,...,n)$. By Riemann Hurwitz
formula, we have
\[  2(g-1)=2(0-1)mn+mn\left(\left(\frac{m-1}{m}\right)+\left(\frac{nm-1}{nm}\right)+
\left(\frac{n-1}{n}\right)\left(\delta+1\right)\right)
\]
Then $\d=\frac{2g+n-1}{m(n-1)}-1$. \\

\textbf{(3)} If $\si$ and $\ta$ both lift to elements of order $mn$ in $G$, then ramification,
$\C=(mn,mn,n,...,n)$. By Riemann Hurwitz formula, we have
\[  2(g-1)=2(0-1)mn+mn\left(\left(\frac{mn-1}{mn}\right)\cdot2 + \left(\frac{n-1}{n}\right)\left(\delta+1\right)\right).
\]
Then $\d=\frac{2g)}{m(n-1)}-1$. \\

\textbf{Case $\G \cong D_{2m}$ :} The ramification of $\p:\bP_x \to \bP_z$ is $(2,2,m)$. i.e. $D_m=\left\langle \si, \ta ,\m| \si^2= \ta^2 = \m^m=1\right\rangle$, where $\m=\si\ta$. Since the branching corresponding to first two ramification points
is the same then there are basically six distinct sub cases which could arise.

\textbf{(4)} If $\si$, $\ta$ and $\m$ lift in $G$ to elements of orders $|\si|$, $|\ta|$ and $|\m|$
respectively, then the ramification is $\C=(2,2,m,n,...,n)$. The dimension $\delta$ can be found
using Riemann Hurwitz formula as follows.
\[  2(g-1)=2(0-1)2mn+2mn\left(\left(\frac{2-1}{2}\right)\cdot2 +\left(\frac{m-1}{m}\right)+ \left(\frac{n-1}{n}\right)\delta\right).
\]
Then, $\d=\frac{g+n-1}{m(n-1)}$.\\

\textbf{(5)} If either $\ta$ or $\m$ lifts in $G$ to element of order $n|\ta|$ or $n|\m|$ then ramification
$\C=(2n,2,m,n,...,n)$. The dimension $\d$ is as follows.
\begin{multline*}
2(g-1)=2(0-1)2mn+\\
2mn\left(\left(\frac{2n-1}{2n}\right)+\left(\frac{2-1}{2}\right) +\left(\frac{m-1}{m}\right)+ \left(\frac{n-1}{n}\right)\delta\right).
\end{multline*}
Then $\d=\frac{2g+m+2n-mn-2}{2m(n-1)}$. \\

\textbf{(6)} If $\m$ lifts to an element of order $n|\m|$ in $G$, then the ramification
$\C=(2,2,mn,n,...,n)$. By Riemann Hurwitz formula, we have
\[   2(g-1)=2(0-1)2mn+2mn\left(\left(\frac{2-1}{2}\right)\cdot2+\left(\frac{mn-1}{mn}\right)+\left(\frac{n-1}{n}\right)\delta\right).
\]
Then $\d=\frac{g}{m(n-1)}.$\\

\textbf{(7)} If both $\si$ and $\ta$ lift to elements of order $n|\si|$ and $n|\ta|$ in $G$, then
ramification $\C=(2n,2n,m,n,...,n)$. By Riemann Hurwitz formula,
we have
\[   2(g-1)=2(0-1)2mn+2mn\left(\left(\frac{2n-1}{2n}\right)\cdot2+\left(\frac{m-1}{m}\right)+\left(\frac{n-1}{n}\right)\delta\right).
\]
Then $\d=\frac{g+m+n-nm-1}{m(n-1)}.$ \\

\textbf{(8)} If both $\si$ and $\m$  lift to elements of order $n|\si|$ and $n|\m|$, then
ramification $\C=(2n,2,mn,n,...,n)$. By Riemann Hurwitz formula,
we have
\begin{multline*}
2(g-1)=2(0-1)2mn+\\
2mn\left(\left(\frac{2n-1}{2n}\right)+\left(\frac{2-1}{2}\right)+\left(\frac{mn-1}{mn}\right)+\left(\frac{n-1}{n}\right)\delta\right).
\end{multline*}
Then $\d=\frac{2g+m-mn}{2m(n-1)}$.\\

\textbf{(9)} If $\si$, $\ta$ and $\m$ lift to elements of orders $n|\si|$, $n|\ta|$ and $n|\m|$
respectively in $G$, then the ramification $\C=(2n,2n,mn,n,...,n)$. Riemann Hurwitz formula gives us
\[   2(g-1)=2(0-1)2mn+2mn\left(\left(\frac{2n-1}{2n}\right)\cdot2+\left(\frac{mn-1}{mn}\right)+\left(\frac{n-1}{n}\right)\delta\right).
\]
Then $\d=\frac{g+m-mn}{m(n-1)}.$\\

\textbf{Case $\G \cong A_4$ :} The ramification of $\p:\bP_x \to \bP_z$ is $(2,3,3)$. i.e. $A_4=\left\langle \si, \ta ,\m| \si^2= \ta^3 = \m^3=1\right\rangle$, where $\m=\si\ta$. Since the branching corresponding to last two ramification points is
the same then there are basically six distinct sub cases which could arise.

\textbf{(10)} If $\si$, $\ta$ and $\m$ lift in $G$ to elements of orders $|\si|$, $|\ta|$ and $|\m|$
respectively, then the ramification $\C=(2,3,3,n,...,n)$. The dimension $\d$ can be found
using Riemann Hurwitz formula as follows.
\[2(g-1)=2(0-1)12n+12n\left(\left(\frac{2-1}{2}\right)+\left(\frac{3-1}{3}\right)\cdot 2
+ \left(\frac{n-1}{n}\right)\delta\right).
\]
Then $\d=\frac{g+n-1}{6(n-1)}$.\\

\textbf{(11)} If $\ta$ lifts in $G$ to element of order $n|\ta|$ then ramification, $\C=(2,3n,3,n,...,n)$. The dimension $\d$ is as follows.
\[   2(g-1)=2(0-1)12n+12n\left(\left(\frac{2n-1}{2n}\right)+\left(\frac{3-1}{3}\right)\cdot2+ \left(\frac{n-1}{n}\right)\delta\right).
\]
Then $\d=\frac{g-n+1}{6(n-1)}$.\\

\textbf{(12)} If $\ta$ and $\m$ lift in $G$ to element of order $n|\ta|$ and $n|\m|$ then ramification $\C=(2,3n,3n,n,...,n)$. The dimension $\d$ is as follows.
\[    2(g-1)=2(0-1)12n+12n\left(\left(\frac{2-1}{2}\right)+\left(\frac{3n-1}{3n}\right)\cdot2+\left(\frac{n-1}{n}\right)\delta\right).
\]
Then $\d=\frac{g-3n+3}{6(n-1)}$. \\

\textbf{(13)} If $\si$ lifts in $G$ to element of order $n|\si|$ then ramification $\C=(2n,3,3,n,...,n)$. The dimension $\delta$ can be found using
Riemann-Hurwitz formula as follows.
\[   2(g-1)=2(0-1)12n+12n\left(\left(\frac{2n-1}{2n}\right)+\left(\frac{3-1}{3}\right)\cdot2+\left(\frac{n-1}{n}\right)\delta\right).
\]
Then $\d=\frac{g-2n+2}{6(n-1)}$. \\

\textbf{(14)} If $\si$ and $\ta$ lift in $G$ to element of order $n|\si|$ and $n|\ta|$ then ramification $\C=(2n,3n,3,n,...,n)$. The dimension $\d$ is as
follows.
\begin{multline*}
2(g-1)=2(0-1)12n+\\
12n\left(\left(\frac{2n-1}{2n}\right)+\left(\frac{3n-1}{3n}\right)+\left(\frac{3-1}{3}\right)+\left(\frac{n-1}{n}\right)\delta\right).
\end{multline*}
Then $\d=\frac{g-4n+4}{6(n-1)}$. \\

\textbf{(15)} If $\si$, $\ta$ and $\m$ lift to elements of orders $n|\si|$, $n|\ta|$ and $n|\m$
respectively in $G$, then ramification $\C=(2n,3n,3n,n,...,n)$. The
dimension $\d$ is as follows.
\[   2(g-1)=2(0-1)12n+12n\left(\left(\frac{2n-1}{2n}\right)+\left(\frac{3n-1}{3n}\right)\cdot2+\left(\frac{n-1}{n}\right)\delta\right).
\]
Then $\d=\frac{g-6n+6}{6(n-1)}$. \\

\textbf{Case $\G \cong S_4$ :} The ramification of $\p:\bP_x \to \bP_z$ is $(2,3,4)$. i.e. $S_4=\left\langle \si, \ta ,\m| \si^2= \ta^3 = \m^4=1 \right\rangle$, where $\m=\si\ta$. Let $\s$ and $\t$ be the lifting of $\si$ and $\ta \in G$ respectively.

\textbf{(16)} If $\si$, $\ta$ and $\m$ lift in $G$ to elements of orders $|\si|$, $|\ta|$ and $|\m|$
respectively, then  ramification $\C=(2,3,4,n,...,n)$. The dimension $\delta$ can find
using Riemann Hurwitz formula as follows.
\begin{multline*}
2(g-1)=2(0-1)24n+\\
24n\left(\left(\frac{2-1}{2}\right)+\left(\frac{3-1}{3}\right)+\left(\frac{4-1}{4}\right)+ \left(\frac{n-1}{n}\right)\delta\right).
\end{multline*}
Then $\d=\frac{g+n-1}{12(n-1)}$.\\

\textbf{(17)} If $\ta$ lifts to an element of order $n|\ta|$ then ramification $(2,3n,4,n,..$ $.,n)$. By Riemann Hurwitz formula
\begin{multline*}
2(g-1)=2(0-1)24n+\\
24n\left(\left(\frac{2-1}{2}\right)+\left(\frac{3n-1}{3n}\right)+\left(\frac{4-1}{4}\right)+ \left(\frac{n-1}{n}\right)\delta\right).
\end{multline*}
Then $\d=\frac{g-3n+3}{12(n-1)}$.\\

\textbf{(18)} If $\m$ lifts to an element of order $n|\m|$, then  ramification $(2,3,4n,n,..$ $.,n)$. By Riemann Hurwitz formula
\begin{multline*}
2(g-1)=2(0-1)24n+\\
24n\left(\left(\frac{2-1}{2}\right)+\left(\frac{3-1}{3}\right)+\left(\frac{4n-1}{4n}\right)+ \left(\frac{n-1}{n}\right)\delta\right).
\end{multline*}
Then $\d=\frac{g-2n+2}{12(n-1)}$.\\

\textbf{(19)} If $\ta$ and $\m$ lift to elements of orders $n|\ta|$ and $n|\m|$ respectively, then  ramification
$(2,3n,4n,n,...,n)$. By Riemann Hurwitz formula
\begin{multline*}
2(g-1)=2(0-1)24n+\\
24n\left(\left(\frac{2-1}{2}\right)+\left(\frac{3n-1}{3n}\right)+\left(\frac{4n-1}{4n}\right)+ \left(\frac{n-1}{n}\right)\delta\right).
\end{multline*}
Then $\d=\frac{g-6n+6}{12(n-1)}$.\\

\textbf{(20)} If $\si$ lifts to an element of order $n|\si|$, then ramification $(2n,3,4,n,...,n)$. By Riemann Hurwitz formula
\begin{multline*}
2(g-1)=2(0-1)24n+\\
24n\left(\left(\frac{2n-1}{2n}\right)+\left(\frac{3-1}{3}\right)+\left(\frac{4-1}{4}\right)+ \left(\frac{n-1}{n}\right)\delta\right).
\end{multline*}
Then $\d=\frac{g-5n+5}{12(n-1)}$.\\

\textbf{(21)} If $\si$ and $\ta$ lift to elements of orders $n|\si|$ and $n|\ta|$ respectively, then  ramification $(2n,3n,4,n,...,n)$. By Riemann Hurwitz formula
\begin{multline*}
2(g-1)=2(0-1)24n+\\
24n\left(\left(\frac{2n-1}{2n}\right)+\left(\frac{3n-1}{3n}\right)+\left(\frac{4-1}{4}\right)+\left(\frac{n-1}{n}\right)\delta\right).
\end{multline*}
Then $\d=\frac{g-9n+9}{12(n-1)}$.\\

\textbf{(22)} If $\si$ and $\m$ lift to elements of orders $n|\si|$ and $n|\m|$ respectively, then ramification $(2n,3,4n,n,...,n)$. By Riemann Hurwitz formula
\begin{multline*}
2(g-1)=2(0-1)24n+\\
24n\left(\left(\frac{2n-1}{2n}\right)+\left(\frac{3-1}{3}\right)+\left(\frac{4n-1}{4n}\right)+\left(\frac{n-1}{n}\right)\delta\right).
\end{multline*}
Then $\d=\frac{g-8n+8}{12(n-1)}$.\\

\textbf{(23)} If $\si$, $\ta$ and $\m$ lift to elements of orders $n|\si|$, $n|\ta|$ and $n|\m|$ respectively, then ramification $(2n,3n,4n,n,...,n)$. By Riemann Hurwitz formula
\begin{multline*}
2(g-1)=2(0-1)24n+\\
24n\left(\left(\frac{2n-1}{2n}\right)+\left(\frac{3n-1}{3n}\right)+\left(\frac{4n-1}{4n}\right)
+ \left(\frac{n-1}{n}\right)\delta\right).
\end{multline*}
Then $\d=\frac{g-12n+12}{12(n-1)}$.\\

\textbf{Case $\G \cong A_5$ :} The ramification of $\p:\bP_x \to \bP_z$ is $(2,3,5)$. i.e. $A_5=\left\langle \si, \ta ,\m| \si^2= \ta^3 = \m^5=1 \right\rangle$, where $\m=\si\ta$. Let $\s$ and $\t$ be the lifting of $\si$ and $\ta \in G$ respectively.

\textbf{(24)} If $\si$, $\ta$ and $\m$ lift to elements of orders $|\si|$, $|\ta|$ and $|\m|$ respectively, then ramification $\C=(2,3,5,n,...,n)$. The dimension $\delta$ can be found using Riemann Hurwitz formula as follows.
\begin{multline*}
2(g-1)=2(0-1)60n+\\
60n\left(\left(\frac{2-1}{2}\right)+\left(\frac{3-1}{3}\right)+\left(\frac{5-1}{5}\right)
+ \left(\frac{n-1}{n}\right)\delta\right).
\end{multline*}
Then $\d=\frac{g+n-1}{30(n-1)}$.\\

\textbf{(25)} If $\m$ lifts to an element of order $n|\m|$, then ramification $(2,3,5n,n,..$ $.,n)$. By Riemann Hurwitz formula
\begin{multline*}
2(g-1)=2(0-1)60n+\\
60n\left(\left(\frac{2-1}{2}\right)+\left(\frac{3-1}{3}\right)+\left(\frac{5n-1}{5n}\right)
+ \left(\frac{n-1}{n}\right)\delta\right).
\end{multline*}
Then $\d=\frac{g-5n+5}{30(n-1)}$.\\

\textbf{(26)} If $\ta$ and $\m$ lift to elements of orders $n|\ta|$ and $n|\m|$ respectively, the ramification $(2,3n,5n,n,...,n)$. By Riemann Hurwitz formula
\begin{multline*}
2(g-1)=2(0-1)60n+\\
60n\left(\left(\frac{2-1}{2}\right)+\left(\frac{3n-1}{3n}\right)+\left(\frac{5n-1}{5n}\right)
+ \left(\frac{n-1}{n}\right)\delta\right).
\end{multline*}
Then $\d=\frac{g-15n+15}{30(n-1)}$.\\

\textbf{(27)} If $\ta$ lifts to an element of order $n|\ta|$, then ramification $(2,3n,5,n,..$ $.,n)$. By Riemann Hurwitz formula
\begin{multline*}
2(g-1)=2(0-1)60n+\\
60n\left(\left(\frac{2-1}{2}\right)+\left(\frac{3n-1}{3n}\right)+\left(\frac{5-1}{5}\right)
+ \left(\frac{n-1}{n}\right)\delta\right).
\end{multline*}
Then $\d=\frac{g-9n+9}{30(n-1)}$.\\

\textbf{(28)} If $\si$ lifts to an element of order $n|\si|$, then ramification $(2n,3,5,n,..$ $.,n)$. By Riemann Hurwitz formula
\begin{multline*}
2(g-1)=2(0-1)60n+\\
60n\left(\left(\frac{2n-1}{2n}\right)+\left(\frac{3-1}{3}\right)+\left(\frac{5-1}{5}\right)
+ \left(\frac{n-1}{n}\right)\delta\right).
\end{multline*}
Then $\d=\frac{g-14n+14}{30(n-1)}$.\\

\textbf{(29)} If $\si$ and $\m$ lift to elements of orders $n|\si|$ and $n|\m|$ respectively, then ramification $(2n,3,5n,n,...,n)$. By Riemann Hurwitz
formula
\begin{multline*}
2(g-1)=2(0-1)60n+\\
60n\left(\left(\frac{2n-1}{2n}\right)+\left(\frac{3-1}{3}\right)+\left(\frac{5n-1}{5n}\right)
+ \left(\frac{n-1}{n}\right)\delta\right).
\end{multline*}
Then $\d=\frac{g-20n+20}{30(n-1)}$.\\

\textbf{(30)} If $\si$ and $\ta$ lift to elements of orders $n|\si|$ and $n|\ta|$ respectively, then ramification $(2n,3n,5,n,...,n)$. By Riemann Hurwitz formula
\begin{multline*}
2(g-1)=2(0-1)60n+\\
60n\left(\left(\frac{2n-1}{2n}\right)+\left(\frac{3n-1}{3n}\right)+\left(\frac{5-1}{5}\right)
+ \left(\frac{n-1}{n}\right)\delta\right).
\end{multline*}
Then $\d=\frac{g-24n+24}{30(n-1)}$.\\

\textbf{(31)} If $\si$, $\ta$ and $\m$ lift to elements of orders $n|\si|$, $n|\ta|$ and $n|\m|$ respectively, then  ramification $(2n,3n,5n,n,...,n)$. By Riemann Hurwitz formula
\begin{multline*}
2(g-1)=2(0-1)60n+\\
60n\left(\left(\frac{2n-1}{2n}\right)+\left(\frac{3n-1}{3n}\right)+\left(\frac{5n-1}{5n}\right)
+ \left(\frac{n-1}{n}\right)\delta\right).
\end{multline*}
Then $\d=\frac{g-30n+30}{30(n-1)}$.\\

\textbf{Case $\G \cong U$ :} The ramification of $\p:\bP_x \to \bP_z$ is $(p^t)$. We know that $(p^t)$ is wildly ramified place; see\cite{VM},Theorem 1. Hence  $\beta=e^*q+q-2$ in equation ~\ref{e2}. Also we know $q=p^t$; see \cite{VM},Theorem 1.

\textbf{(32)} If element of order $p^t$ lifts to an element of order $p^t$, then ramification $\C=(p^t,n,...,n)$. We know $e^*=1$; see \cite{VM},Theorem 1. The dimension $\delta$ can be found using Riemann Hurwitz formula as follows.
\[    2(g-1)=2(0-1)np^t+np^t\left(\left(\frac{p^t+p^t-2}{p^t}\right)+ \left(\frac{n-1}{n}\right)\left(\delta+2\right)\right).
\]
Then $\d=\frac{2g+2n-2}{p^t(n-1)}-2$.\\

\textbf{(33)} If element of order $p^t$ lifts to an element of order $np^t$, then ramification $\C=(np^t,n,...,n)$. In this case $e^*=n.$ Also $(n,p)=1$ and $n|p^t-1$. The dimension $\delta$ can be found using Riemann Hurwitz formula as follows.
\[    2(g-1)=2(0-1)np^t+np^t\left(\left(\frac{np^t+p^t-2}{np^t}\right)+ \left(\frac{n-1}{n}\right)\left(\delta+2\right)\right).
\]
Then $\d=\frac{2g+np^t-p^t}{p^t(n-1)}-2$.\\

\textbf{Case $\G \cong K_m$ :} The ramification of $\p:\bP_x \to \bP_z$ is $(mp^t,m)$. We know that the first place is wildly ramified; see \cite{VM},Theorem 1. Hence $\beta_1=e^*_1q_1+q_1-2$ in equation ~\ref{e2}. We know $q_1=p^t$;  see \cite{VM},Theorem 1.

\textbf{(34)} If both elements of orders $mp^t$ and $m$ lift to elements of orders $mp^t$ and $m$ respectively, then ramification $\C=(mp^t,m,n,...,n)$. We know $e_1^*=m$, $m|p^t-1$ and $(m,p)=1$; see \cite{VM},Theorem 1. Riemann Hurwitz formula gives us,
\begin{multline*}
2(g-1)=2(0-1)nmp^t+\\
nmp^t\left(\left(\frac{mp^t+p^t-2}{mp^t}\right)+\left(\frac{m-1}{m}\right)+ \left(\frac{n-1}{n}\right)\left(\delta+1\right)\right).
\end{multline*}
Then $\d=\frac{2g+2n-2}{mp^t(n-1)}-1$.\\

\textbf{(35)} If element of order $m$ lifts to an element of order $nm$, then ramification $\C=(mp^t,nm,n,...,n)$. As in previous case $e_1^*=m$, $m|p^t-1$ and $(m,p)=1$. The dimension $\delta$ can be found using Riemann Hurwitz formula as follows.
\begin{multline*}
2(g-1)=2(0-1)nmp^t+\\
nmp^t\left(\left(\frac{mp^t+p^t-2}{mp^t}\right)+\left(\frac{nm-1}{nm}\right)+ \left(\frac{n-1}{n}\right)\left(\delta+1\right)\right).
\end{multline*}
Then $\d=\frac{2g+2n+p^t-np^t-2}{mp^t(n-1)}-1$.\\

\textbf{(36)} If element of order $mp^t$ lifts to an element of order $nmp^t$, then ramification $\C=(nmp^t,m,n,...,n)$. In this case $e_1^*=nm.$ Also $(nm,p)=1$ and $nm|p^t-1$. The dimension $\delta$ can be found using Riemann Hurwitz formula as follows.
\begin{multline*}
2(g-1)=2(0-1)nmp^t+\\
nmp^t\left(\left(\frac{nmp^t+p^t-2}{nmp^t}\right)+\left(\frac{m-1}{m}\right)+ \left(\frac{n-1}{n}\right)\left(\delta+1\right)\right).
\end{multline*}
Then $\d=\frac{2g+np^t-p^t}{mp^t(n-1)}-1$.\\

\textbf{(37)} If both elements of orders $mp^t$ and $m$ lift to elements of orders $nmp^t$ and $nm$ respectively, then ramification $\C=(nmp^t,nm,n,...,n)$. As in previous case $e_1^*=nm$, $(nm,p)=1$ and $nm|p^t-1$. Riemann Hurwitz formula gives us,
\begin{multline*}
2(g-1)=2(0-1)nmp^t+\\
nmp^t\left(\left(\frac{nmp^t+p^t-2}{nmp^t}\right)+\left(\frac{nm-1}{nm}\right)+ \left(\frac{n-1}{n}\right)\left(\delta+1\right)\right).
\end{multline*}
Then $\d=\frac{2g}{mp^t(n-1)}-1$.\\

\textbf{Case $\G \cong PSL_2(q)$ :} The ramification of $\p:\bP_x \to \bP_z$ is $(\alpha,\beta)$, where $\alpha=\frac{q(q-1)}{2}$ and $\beta=\frac{q+1}{2}$. We know that the first place is wildly ramified; see \cite{VM},Theorem 1. Hence $\beta_1=e^*_1q_1+q_1-2$ in equation ~\ref{e2}. We know $q_1=q$;  see \cite{VM},Theorem 1. Let $m$ be the size of $PSL_2(q)$. i.e. $m=\frac{q(q-1)(q+1)}{2}$.

\textbf{(38)} If both elements of orders $\alpha$ and $\beta$ lift to elements of orders $\alpha$ and $\beta$ respectively, then ramification $\C=(\alpha,\beta,n,...,n)$. We know $e_1^*=\frac{q-1}{2}$ and $\left(\frac{q-1}{2},p\right)=1$; see \cite{VM},Theorem 1. Riemann Hurwitz formula gives us,
\begin{multline*}
2(g-1)=2(0-1)nm+\\
nm\left(\left(\frac{\frac{q(q-1)}{2}+q-2}{\frac{q(q-1)}{2}}\right)+\left(\frac{\frac{q+1}{2}-1}{\frac{q+1}{2}}\right)+ \left(\frac{n-1}{n}\right)\left(\delta+1\right)\right).
\end{multline*}
Then $\d=\frac{2g+2n-2}{m(n-1)}-1$.\\

\textbf{(39)} If an element of order $\beta$ lifts to element of order $n\beta$, then ramification $\C=(\alpha,n\beta,n,...,n)$. As in previous case $e_1^*=\frac{q-1}{2}$ and $\left(\frac{q-1}{2},p\right)=1$. By using Riemann Hurwitz formula we can find $\d$ as follows.
\begin{multline*}
2(g-1)=2(0-1)nm+\\
nm\left(\left(\frac{\frac{q(q-1)}{2}+q-2}{\frac{q(q-1)}{2}}\right)+\left(\frac{\frac{n(q+1)}{2}-1}{\frac{n(q+1)}{2}}\right)+ \left(\frac{n-1}{n}\right)\left(\delta+1\right)\right).
\end{multline*}
Then $\d=\frac{2g+q(q-1)-n(q+1)(q-2)-2}{m(n-1)}-1$.\\

\textbf{(40)} If an element of order $\alpha$ lifts to element of order $n\alpha$, then ramification $\C=(n\alpha,\beta,n,...,n)$. In this case $e_1^*=\frac{n(q-1)}{2}$, $\frac{n(q-1)}{2}|q-1$ and $\left(\frac{n(q-1)}{2},p\right)=1$. So only possible values for $n$ are 1 and 2. By using Riemann Hurwitz formula we can find $\d$ as follows.
\begin{multline*}
2(g-1)=2(0-1)nm+\\
nm\left(\left(\frac{\frac{nq(q-1)}{2}+q-2}{\frac{nq(q-1)}{2}}\right)+\left(\frac{\frac{q+1}{2}-1}{\frac{q+1}{2}}\right)+ \left(\frac{n-1}{n}\right)\left(\delta+1\right)\right).
\end{multline*}
Then $\d=\frac{2g+nq(q-1)+q-q^2}{m(n-1)}-1$.\\

\textbf{(41)} If both elements of orders $\alpha$ and $\beta$ lift to elements of orders $n\alpha$ and $n\beta$ respectively, then ramification $\C=(n\alpha,n\beta,n,...,n)$. As in previous case $e_1^*=\frac{n(q-1)}{2}$, $\left(\frac{n(q-1)}{2},p\right)=1$ and $n$ can be either 1 or 2. Riemann Hurwitz formula gives us,
\begin{multline*}
2(g-1)=2(0-1)nm+\\
nm\left(\left(\frac{\frac{nq(q-1)}{2}+q-2}{\frac{nq(q-1)}{2}}\right)+\left(\frac{\frac{n(q+1)}{2}-1}{\frac{n(q+1)}{2}}\right)+ \left(\frac{n-1}{n}\right)\left(\delta+1\right)\right).
\end{multline*}
Then $\d=\frac{2g}{m(n-1)}-1$.\\

\textbf{Case $\G \cong PGL_2(q)$ :} The ramification of $\p:\bP_x \to \bP_z$ is $(2\alpha,2\beta)$, where $\alpha$ and $\beta$ as in the case $PSL_2(q)$. We know that the first place is wildly ramified; see \cite{VM},Theorem 1. Hence $\beta_1=e^*_1q_1+q_1-2$ in equation ~\ref{e2}. Also we know $q_1=q$; see \cite{VM},Theorem 1. Let $m$ be the size of $PGL_2(q)$. i.e. $m=q(q-1)(q+1)$.

\textbf{(42)} If both elements of orders $2\alpha$ and $2\beta$ lift to elements of orders $2\alpha$ and $2\beta$ respectively, then ramification $\C=(2\alpha,2\beta,n,...,n)$. We know $e_1^*=q-1$ and $(q-1,p)=1$; see \cite{VM},Theorem 1. Riemann Hurwitz formula gives us,
\begin{multline*}
2(g-1)=2(0-1)nm+\\
nm\left(\left(\frac{q(q-1)+q-2}{q(q-1)}\right)+\left(\frac{q+1-1}{q+1}\right)+ \left(\frac{n-1}{n}\right)\left(\delta+1\right)\right).
\end{multline*}
Then $\d=\frac{2g+2n-2}{m(n-1)}-1$.\\

\textbf{(43)} If an element of order $2\beta$ lifts to element of order $2n\beta$, then ramification $\C=(2\alpha,2n\beta,n,...,n)$. As in previous case $e_1^*=q-1$ and $(q-1,p)=1$. By using Riemann Hurwitz formula we can find $\d$ as follows.
\begin{multline*}
2(g-1)=2(0-1)nm+\\
nm\left(\left(\frac{q(q-1)+q-2}{q(q-1)}\right)+\left(\frac{n(q+1)-1}{n(q+1)}\right)+ \left(\frac{n-1}{n}\right)\left(\delta+1\right)\right).
\end{multline*}
Then $\d=\frac{2g+q(q-1)-n(q+1)(q-2)-2}{m(n-1)}-1$.\\

\textbf{(44)} If an element of order $2\alpha$ lifts to element of order $2n\alpha$, then ramification $\C=(2n\alpha,2\beta,n,...,n)$. In this case $e_1^*=n(q-1)$, $\left(n(q-1),p\right)=1$ and $n(q-1)|q-1$. So only possible value for $n$ is 1. By using Riemann Hurwitz formula we can find $\d$ as follows.
\begin{multline*}
2(g-1)=2(0-1)nm+\\
nm\left(\left(\frac{nq(q-1)+q-2}{nq(q-1)}\right)+\left(\frac{q+1-1}{q+1}\right)+ \left(\frac{n-1}{n}\right)\left(\delta+1\right)\right).
\end{multline*}
Then $\d=\frac{2g+nq(q-1)+q-q^2}{m(n-1)}-1$.\\

\textbf{(45)} If both elements of orders $2\alpha$ and $2\beta$ lift to elements of orders $2n\alpha$ and $2n\beta$ respectively, then ramification $\C=(2n\alpha,2n\beta,n,...,n)$. As in previous case $e_1^*=n(q-1)$, $\left(n(q-1),p\right)=1$ and $n=1$. Riemann Hurwitz formula gives us,
\begin{multline*}
2(g-1)=2(0-1)nm+\\
nm\left(\left(\frac{nq(q-1)+q-2}{nq(q-1)}\right)+\left(\frac{n(q+1)-1}{n(q+1)}\right)+ \left(\frac{n-1}{n}\right)\left(\delta+1\right)\right).
\end{multline*}
Then $\d=\frac{2g}{m(n-1)}-1$.\\

This completes the proof.
\end{proof}

Now we determine the automorphism group $G$ for each $\G$ as separate theorems. We know that each $\G$ has sub cases. So we list $G$ for each sub cases under the appropriate theorem. In some cases we give a presentation for $G$.

\begin{remark} \label{r1}
Let $\G$ be a group such that $s \in \G$ and $s^m=1$. Let $C_n$ be the cyclic group of order $n$ and $r$ be the generator of it. Let $G$ be a extension of $\G$ by $C_n$ such that $C_n \triangleleft G$. Then $srs^{-1}=r^l$, where $(l,n)=1$ and $l^m \equiv 1$ (mod n).
\end{remark}

\begin{proof}
Since $C_n \triangleleft G$, $srs^{-1}=r^l$ for some $1\leq l\leq n$. But $(srs^{-1})^n=1$. Hence $(l,n)=1$. Since $\s^m\r\s^{-m}=\r$ and $\s^m\r\s^{-m}=\r^{l^m}$, $l^m\equiv 1$ (mod n).
\end{proof}

\subsubsection{$\G \cong C_m$}

\begin{thm}\label{th9}
The automorphism group G of a cyclic curve of genus $g\geq2$ with $\G \cong C_m$ is as follows.

(1) If G has ramification as in case 1, then there are two sub-cases. If $m=1$ then $G \cong C_n$, otherwise G has a presentation:
\begin{center}
$\left\langle \r, \s \right|\r^n=1,\s^m=1,\s\r\s^{-1}=\r^l \rangle$
\end{center}
where (l,n)=1 and $l^m\equiv 1$ (mod n). But if $(m,n)=1$, then $l=n-1$.

(2) If G has ramification as in cases 2-3 in Table ~\ref{t2} then $G \cong C_{mn}$.

\end{thm}

\begin{proof}
We know that the second cohomology group is as follows; see Table 1 in \cite{AK}.
\begin{equation*}
H^2(C_m,C_{n}) \cong C_{(n,m)}
\end{equation*}

(1) If $m=1$, then $G$ is a cyclic extension of $C_n$ by $C_1$. $C_n$ is the only one possible extension. Now we consider the sub case $m>1$. Let $C_n=\left\langle \r|\r^n=1\right\rangle$ and let $C_m=\left\langle \si| \si^m=1\right\rangle$. Let $\s$ be the lifting $\si$ in $G$. In the case 1, an element of order $m$ lifts to an element of order $m$ in $G$. Hence $\s^m=1$. Since $C_n \triangleleft G$, $\s\r\s^{-1}=\r^l$ for some $l \in \{1,...,n\}$. By Remark ~\ref{r1}, $(l,n)=1$ and $l^m\equiv 1$ (mod n). Hence $G$ has a presentation:
\begin{center}
$\left\langle \r, \s \right|\r^n=1,\s^m=1,\s\r\s^{-1}=\r^l \rangle$
\end{center}
If $(m,n)=1$, then $\left| H^2(C_m,C_{n})\right|=1$. Hence there is only one extension. If $l=1$, $G \cong C_m \times C_n=C_{mn}$. Since this case $G$ does not have an element of order $mn$, $l\neq 1$. So if $(m,n)=1$ then $l=n-1$.

(2) If $G$ has ramification as in cases 2-3 in Table ~\ref{t2} then $G$ has an element of order $mn$. Among the extensions $C_m$ by $C_n$, $C_{mn}$ is the only one extension for which has an element of order $mn$. Hence, for those cases $G \cong C_{mn}$.
\end{proof}

\subsubsection{$\G \cong D_{2m}$}

\begin{thm}\label{th10}
Let $\G=G/C_n \cong D_{2m}$. The automorphism group G is as follows.

(1) If n is odd then $G \cong D_{2m} \times C_n$.

(2) If n is even and m is odd then $G \cong D_{2m} \times C_n$ for the cases 4,6 and $G \cong G_9$ for the cases 7,9 in Table ~\ref{t2} respectively, where $G_{9}$ is as follows.
\begin{align*}
\begin{split}
G_{9}=& \left\langle \r, \s, \t \right|\r^n=1,\s^2=\r,\t^2=\r^{n-1},(\s\t)^m=\r^{\frac{n}{2}},\s\r\s^{-1}=\r,\t\r\t^{-1}=\r \rangle\\
\end{split}
\end{align*}
There are no possible group extensions for the cases 5 and 8 in Table ~\ref{t2}.

(3) If n is even and m is even then $G \cong G_{4},G_{5},G_{6},G_{7},G_{8},G_{9}$ for the cases 4-9 in Table ~\ref{t2} respectively, where $G_{4} - G_{9}$  are as follows.
\begin{align*}
\begin{split}
G_{4}=& D_{2m} \times C_n \\
G_{5}=& \left\langle \r, \s, \t \right|\r^n=1,\s^2=\r,\t^2=1,(\s\t)^m=1,\s\r\s^{-1}=\r,\t\r\t^{-1}=\r^{n-1} \rangle\\
G_{6}=& D_{2mn} \\
G_{7}=& \left\langle \r, \s, \t \right|\r^n=1,\s^2=\r,\t^2=\r^{n-1},(\s\t)^m=1,\s\r\s^{-1}=\r,\t\r\t^{-1}=\r \rangle\\
G_{8}=& \left\langle \r, \s, \t \right|\r^n=1,\s^2=\r,\t^2=1,(\s\t)^m=\r^{\frac{n}{2}},\s\r\s^{-1}=\r,\t\r\t^{-1}=\r^{n-1} \rangle \\
G_{9}=& \left\langle \r, \s, \t \right|\r^n=1,\s^2=\r,\t^2=\r^{n-1},(\s\t)^m=\r^{\frac{n}{2}},\s\r\s^{-1}=\r,\t\r\t^{-1}=\r \rangle
\end{split}
\end{align*}
\end{thm}

\begin{proof}
We know that the second cohomology group is as follows; see Table 1 in \cite{AK}.
\begin{equation*}
H^2(D_{2m},C_{n}) \cong
\begin{cases}
1 & \text{if $(n,2)=1$}\\
C_{2} & \text{if $(n,2)=2$ and $(m,2)=1$}\\
C_2 \times C_2 \times C_2 & \text{ if $(n,2)=2$ and $(m,2)=2$}
\end{cases}
\end{equation*}

(1) If $n$ is odd then $\left|H^2(D_{2m},C_{n})\right|=1$. Hence $G \cong D_{2m} \times C_n$.

(2) If $n$ is even and $m$ is odd then $|H^2(D_{2m},C_{n})| =2$. So there are at most 2 extensions which could occur. For cases 4 and 6 $G \cong D_{2m} \times C_n$ because in those cases two elements of order 2 of $G$ lift to elements of same order. In cases 7 and 9, two elements of order 2 left to elements of order $2n$. Let $C_n=\left\langle \r|\r^n=1\right\rangle$. The group $D_{2m}$ has a presentation, $\left\langle \si, \ta ,\m| \si^2= \ta^2 = \m^m=1 \right\rangle$, where $\m=\si\ta$. Let $\s$ and $\t$ be the lifting of $\si$ and $\ta$ in $G$ respectively and we know that $C_n \triangleleft D_{2m}$. Hence $\s\r\s^{-1}=\r^l$ and $\t\r\t^{-1}=r^k$.  By Remark ~\ref{r1}, $(l,n)=1$, $l^2\equiv 1$ (mod $n$) and $(k,n)=1$, $k^2\equiv 1$ (mod $n$). We choose $k=l=1$. Since both $\si$ and $\ta$ lift to elements of order $n|\si|$ and $n|\ta|$ in $G$, then we choose $\s^{2}=\r$ and $\t^2=\r^{n-1}$, because both $\r$ and $\r^{n-1}$ have order $n$ in $C_n$. In case 9, $\m$ lifts to element of order $n|\m|$. Thus  we choose $(\s\t)^m=\r^{\frac{n}{2}}$, because $\r^{\frac{n}{2}}$ has order $2$ in $C_n$. Hence $G \cong G_9$, where $G_9$ is as follows.
\begin{align*}
\begin{split}
G_{9}=& \left\langle \r, \s, \t \right|\r^n=1,\s^2=\r,\t^2=\r^{n-1},(\s\t)^m=\r^{\frac{n}{2}},\s\r\s^{-1}=\r,\t\r\t^{-1}=\r \rangle\\
\end{split}
\end{align*}
we know that the dimension $\d$ is an integer. But $\d$'s of cases 5 and 8 cannot be an integer when $n$ is even and $m$ is odd. So there are no possible automorphism groups for these cases.

(3) If $n$ and $m$ both even then $|H^2(D_{2m},C_{n})| =6$.  So there are at most 6 extensions which could occur. As in proof of part (2), $D_{2m}=\left\langle \si, \ta ,\m| \si^2= \ta^2 = \m^m=1 \right\rangle$, where $\m=\si\ta$, $C_n=\left\langle \r|\r^n=1\right\rangle$ and $\s$ and $\t$ be the lifting of $\si$ and $\ta$ in $G$ respectively. Also, $\s\r\s^{-1}=\r^l$ and $\t\r\t^{-1}=r^k$. By Remark ~\ref{r1}, $(l,n)=1$, $l^2\equiv 1$ (mod $n$) and $(k,n)=1$, $k^2\equiv 1$ (mod $n$). We choose $k=1$.

In case 4, $\si$, $\ta$ and $\m$ lift in $G$ to elements of orders $|\si|$, $|\ta|$ and $|\m|$ respectively. Hence $G \cong D_{2m} \times C_n$.

If $\ta$ lifts to element of order $n|\ta|$ like in case 5, then we choose $\s$ such that $\s^2=\r$, because order of $\r$ is $n$ in $c_n$. Since other two generators lift to elements of same orders that they had before, $\t^2=1$ and $(\s\t)^m=1$. Further we choose $l=n-1$. So $G$ is isomorphic to $G_5$.

In case 6, $\m$ lifts to an element of order $n|\m|$ in $G$. Hence $G \cong D_{2mn}$.

In case 7, both $\si$ and $\ta$ lift to elements of orders $n|\si|$ and $n|\ta|$ in $G$, then we choose $\s^{2}=\r$ and $\t^2=\r^{n-1}$, because both $\r$ and $\r^{n-1}$ have order $n$ in $C_n$. Since other generator lifts to an element of same order that it has before, $(\s\t)^m=1$. Also we choose $l=1$. So $G$ is isomorphic to $G_7$.

If both $\si$ and $\m$  lift to elements of orders $n|\si|$ and $n|\m|$ like in case 8, then we choose $\s^{2}=\r$ and $(\s\t)^{m}=\r^{\frac{n}{2}}$. Since the order of $\ta$ is remaining the same, $\t^2=1$. Further we choose $l=n-1$. Hence $G \cong G_{8}$.

In case 9, $\si$, $\ta$ and $\m$ lift to elements of orders $n|\si|$, $n|\ta|$ and $n|\m|$ respectively, then  we choose $\s$ and $\t$ such that $\s^2=\r$, $\t^2=\r^{n-1}$ and $(\s\t)^m=\r^{\frac{n}{2}}$. Also we choose $l=1$. Hence $G \cong G_{9}$.

\end{proof}

\subsubsection{$\G \cong A_4$}

\begin{lemma}\label{l4}
Let $G$ be a group extension of $A_4$ by $C_n$ and let $n=p_1^{\alpha_1}p_2^{\alpha_2}...p_r^{\alpha_r}$. If 3 $\nmid (p_i-1)$ for all $1\leq i \leq r$, then G is a central extension.
\end{lemma}

\begin{proof}
Let's consider the conjugation action of $A_4$ on $C_n$ and the homomorphism $\gamma:A_4\longrightarrow Aut(C_n)$. Then $im(\gamma) \in 1, C_3, A_4$. If $n=p_1^{\alpha_1}p_2^{\alpha_2}...p_r^{\alpha_r}$ and 3 $\nmid (p_i-1)$ for all $1\leq i \leq r$, then 3 $\nmid\left|Aut(C_n)\right|$. So $3 \nmid\left|im(\gamma)\right|$. i.e. $\left|im(\gamma)\right|$=1. Hence $G$ is central extension of $A_4$ by $C_n$.
\end{proof}

\begin{thm}\label{th7}
Let $\X_g$ be a genus $g\geq 2$ cyclic curve with $\G \cong A_4$. Then $G:=Aut(\X_g)$ as follows.

(1) If n is odd and not a multiple of 3 then $G \cong A_4 \times C_n$.

(2) If n is odd and a multiple of 3 then $G \cong G'_{10},G'_{12},G'_{13},G'_{15}$ for the cases 10,12,13,15 in Table ~\ref{t2} respectively, where $G'_{10},G'_{12},G'_{13},G'_{15}$ are as follows.
\begin{align*}
\begin{split}
G'_{10}=& \left\langle \r, \s, \t \right|\r^n=1,\s^2=1,\t^3=1,(\s\t)^3=1,\s\r\s^{-1}=\r,\t\r\t^{-1}=\r^l \rangle\\
G'_{12}=& \left\langle \r, \s, \t \right|\r^n=1,\s^2=1,\t^3=\r^{\frac{n}{3}},(\s\t)^3=\r^{\frac{n}{3}},\s\r\s^{-1}=\r,\t\r\t^{-1}=\r^l \rangle\\
G'_{13}=& \left\langle \r, \s, \t \right|\r^n=1,\s^2=\r^{\frac{n}{3}},\t^3=1,(\s\t)^3=1,\s\r\s^{-1}=\r,\t\r\t^{-1}=\r^l \rangle\\
G'_{15}=& \left\langle \r, \s, \t \right|\r^n=1,\s^2=\r^{\frac{n}{3}},\t^3=\r^{\frac{n}{3}},(\s\t)^3=\r^{\frac{n}{3}},\s\r\s^{-1}=\r,\t\r\t^{-1}=\r^l \rangle
\end{split}
\end{align*}
where $(l,n)=1$ and $l^3\equiv 1$ (mod n).
Furthermore $G'_{10}\cong G'_{13}$, $G'_{12} \cong G'_{15}$ and there are no possible group extensions for the cases 11, 14 in Table ~\ref{t2}.

(3) If n is even, not a multiple of 3, then if $n$ satisfies the condition in Lemma ~\ref{l4} then $G \cong A_4 \times C_n$ when $G$ has ramification as in cases 10 and $G$ has ramification as in cases 11-15 in Table ~\ref{t2} then $G$ has a presentation:
\begin{center}
$\left\langle \r, \s, \t \right|\r^n=1,\s^2=\r^{\frac{n}{2}},\t^3=\r^{\frac{n}{2}},(\s\t)^3=\r^{\frac{n}{2}},\s\r\s^{-1}=\r,\t\r\t^{-1}=\r \rangle $
\end{center}
If $n$ does not satisfy the condition in Lemma ~\ref{l4}, then $G \cong G_{10},G_{11},G_{12},G_{13}$, $G_{14},G_{15}$ for the cases 10-15 in Table ~\ref{t2} respectively, where $G_{10} - G_{15}$ are as in (4).

(4) If n is even and multiple of 3 then $G \cong G_{10},G_{11},G_{12},G_{13},G_{14},G_{15}$ for the cases 10-15 in Table ~\ref{t2} respectively, where $G_{10} - G_{15}$ are as follows.
\begin{align*}
\begin{split}
G_{10}=& \left\langle \r, \s, \t \right|\r^n=1,\s^2=1,\t^3=1,(\s\t)^3=1,\s\r\s^{-1}=\r,\t\r\t^{-1}=\r^k \rangle\\
G_{11}=& \left\langle \r, \s, \t \right|\r^n=1,\s^2=1,\t^3=\r^{\frac{n}{2}},(\s\t)^3=1,\s\r\s^{-1}=\r,\t\r\t^{-1}=\r^k \rangle\\
G_{12}=& \left\langle \r, \s, \t \right|\r^n=1,\s^2=1,\t^3=\r^{\frac{n}{2}},(\s\t)^3=\r^{\frac{n}{2}},\s\r\s^{-1}=\r,\t\r\t^{-1}=\r^k \rangle\\
G_{13}=& \left\langle \r, \s, \t \right|\r^n=1,\s^2=\r^{\frac{n}{2}},\t^3=1,(\s\t)^3=1,\s\r\s^{-1}=\r,\t\r\t^{-1}=\r^k \rangle\\
G_{14}=& \left\langle \r, \s, \t \right|\r^n=1,\s^2=\r^{\frac{n}{2}},\t^3=\r^{\frac{n}{2}},(\s\t)^3=1,\s\r\s^{-1}=\r,\t\r\t^{-1}=\r^k \rangle \\
G_{15}=& \left\langle \r, \s, \t \right|\r^n=1,\s^2=\r^{\frac{n}{2}},\t^3=\r^{\frac{n}{2}},(\s\t)^3=\r^{\frac{n}{2}},\s\r\s^{-1}=\r,\t\r\t^{-1}=\r^k \rangle
\end{split}
\end{align*}
where $(k,n)=1$ and $k^3\equiv 1$ (mod n).
Furthermore $G_{10}\cong G_{11} \cong G_{12}$ and $G_{13} \cong G_{14} \cong G_{15}$.
\end{thm}

\begin{proof}
We know that the second cohomology group is as follows; see Table 1 in \cite{AK}.
\begin{equation*}
H^2(A_{4},C_{n}) \cong C_{(n,2)} \times C_{(n,3)}
\end{equation*}

(1) If $n$ is not a multiple of $3$ then $H^2(A_{4},C_{n}) =C_{(n,2)} \times C_1$. If we consider the case that $n$ is odd under the condition $n$ is not a multiple of 3, then $\left|H^2(A_{4},C_{n})\right|=1$. Hence $G \cong A_4 \times C_n$.

(2) If $n$ is odd and a multiple of 3 then $\left|H^2(A_{4},C_{n})\right|=3$. So there are at most 3 extensions which could occur. Let $C_n=\left\langle \r|\r^n=1\right\rangle$. The group $A_4$ has a presentation, $\left\langle \si, \ta ,\m| \si^2= \ta^3 = \m^3=1 \right\rangle$, where $\m=\si\ta$. Let $\s$ and $\t$ be the lifting of $\si$ and $\ta$ in $G$ respectively and we know that $C_n \triangleleft A_4$. Hence $\s\r\s^{-1}=\r^k$ and $\t\r\t^{-1}=r^l$. We choose $k=1$. By Remark ~\ref{r1}, $(l,n)=1$ and $l^3\equiv 1$ (mod $n$).

The case 10 in Table ~\ref{t2} is lifting of  $\si$, $\ta$ and $\m$ to elements of orders $|\si|$, $|\ta|$ and $|\m|$ respectively, then $\s^2=1,$ $\t^3=1$ and $(\s\t)^3=1$. Hence $G$ has a presentation as in $G'_{10}$.

If $\ta$ and $\m$ lift to elements of orders $n|\ta|$ and $n|\m|$ respectively, then we choose $\s$ and $\t$ such that $\t^3=\r^{\frac{n}{2}}$ and $(\s\t)^3=\r^{\frac{n}{2}}$. In case 12, we have such a situation. Since the order of $\si$ is remaining the same, $\s^2=1$. So $G$ has presentation as in $G'_{12}$.

In case 13, $\si$ lifts to an element of order $n|\si|$, then we choose $\s$ such that $\s^2=\r^{\frac{n}{2}}$. Since the orders of $\ta$ and $\m$ are remaining the same, $\t^3=1$ and $(\s\t)^3=1$. So $G \cong G'_{13}$.

In case 15, $\si$, $\ta$ and $\m$ lift to elements of orders $n|\si|$, $n|\ta|$ and $n|\m|$ respectively, then  we choose $\s$ and $\t$ such that $\s^2=\r^{\frac{n}{2}}$, $\t^3=\r^{\frac{n}{2}}$ and $(\s\t)^3=\r^{\frac{n}{2}}$. Hence $G \cong G'_{15}$.

(3) If $n$ is even and not a multiple of 3 then  $\left|H^2(A_{4},C_{n})\right|=2$. By Lemma ~\ref{l4}, if $n=p_1^{\alpha_1}p_2^{\alpha_2}...p_r^{\alpha_r}$ and $3\nmid (p_i-1)$ for all $1\leq i \leq r$ then $G$ is a central extension. Hence there are two extensions. So $G \cong A_4 \times C_n$ for the cases 10 in Table ~\ref{t2}, because $A_4 \times C_n$ does not have element of order $2n$. By using GAP algebra package we found out that $G:=\left\langle \r, \s, \t \right|\r^n=1,\s^2=\r^{\frac{n}{2}},\t^3=\r^{\frac{n}{2}},(\s\t)^3=\r^{\frac{n}{2}},\s\r\s^{-1}=\r,\t\r\t^{-1}=\r \rangle$ for the cases 11-15 in table ~\ref{t2}. If $n$ does not satisfy the condition in Lemma ~\ref{l4} then $G$ isomorphic to $G_{10}-G_{15}$ in (4) and proof is exactly similar to proof in (4).

(4) If $n$ is even then $\left|H^2(A_{4},C_{n})\right|=2$ or $6$. So there are at most 6 extensions which could occur. Let $C_n=\left\langle \r|\r^n=1\right\rangle$. As proof of part (2), $A_4$ has a presentation, $\left\langle \si, \ta ,\m| \si^2= \ta^3 = \m^3=1 \right\rangle$, where $\m=\si\ta$. Let $\s$ and $\t$ be the lifting of $\si$ and $\ta$ in $G$ respectively and we know that $C_n \triangleleft A_4$. Hence $\s\r\s^{-1}=\r^l$ and $\t\r\t^{-1}=r^k$. We choose $l=1$. By Remark ~\ref{r1}, $(k,n)=1$ and $k^3\equiv 1$ (mod $n$).

The case 10 in Table ~\ref{t2} is lifting of  $\si$, $\ta$ and $\m$ to elements of orders $|\si|$, $|\ta|$ and $|\m|$ respectively, then $\s^2=1,$ $\t^3=1$ and $(\s\t)^3=1$. Hence $G$ has a presentation as in $G_{10}$.

In case 11, $\ta$ lifts to an element of order $n|\ta|$, we choose $\t$ such that $\t^3=\r^{\frac{n}{2}}$. Since the orders of $\si$ and $\m$ are remaining the same, $\s^2=1$ and $(\s\t)^3=1$. Hence $G \cong G_{11}$.

If $\ta$ and $\m$ lift to elements of orders $n|\ta|$ and $n|\m|$ respectively, then we choose $\s$ and $\t$ such that $\t^3=\r^{\frac{n}{2}}$ and $(\s\t)^3=\r^{\frac{n}{2}}$. In case 12, we have such a situation. Since the order of $\si$ is remaining the same, $\s^2=1$. So $G$ has presentation as in $G_{12}$.

In case 13, $\si$ lifts to an element of order $n|\si|$, then we choose $\s$ such that $\s^2=\r^{\frac{n}{2}}$. Since the orders of $\ta$ and $\m$ are remaining the same, $\t^3=1$ and $(\s\t)^3=1$. So $G \cong G_{13}$.

If $\si$ and $\ta$ lift to elements of orders $n|\si|$ and $n|\ta|$ respectively like in case 14, then we choose $\s$ and $\t$ such that $\s^2=\r^{\frac{n}{2}}$ and $\t^3=\r^{\frac{n}{2}}$. Since the order of $\m$ does not change $(\s\t)^3=1$. Hence $G \cong G_{14}$.

In case 15, $\si$, $\ta$ and $\m$ lift to elements of orders $n|\si|$, $n|\ta|$ and $n|\m|$ respectively, then  we choose $\s$ and $\t$ such that $\s^2=\r^{\frac{n}{2}}$, $\t^3=\r^{\frac{n}{2}}$ and $(\s\t)^3=\r^{\frac{n}{2}}$. Hence $G \cong G_{15}$.
\end{proof}

\subsubsection{$\G \cong S_4$}

\begin{thm} \label{th8}
The full automorphism groups for the cases 16-23 in Table ~\ref{t2} as follows.

(1) If $n$ is odd then $G \cong S_4 \times C_n$.

(2) If $n$ is even then $G \cong G_{16},G_{17},G_{18},G_{19},G_{20},G_{21},G_{22},G_{23}$ for the cases 16-23 in Table ~\ref{t2} respectively, where $G_{16} - G_{23}$ are as follows.
\begin{align*}
\begin{split}
G_{16}=& \left\langle \r, \s, \t \right|\r^n=1,\s^2=1,\t^3=1,(\s\t)^4=1,\s\r\s^{-1}=\r^l,\t\r\t^{-1}=\r \rangle\\
G_{17}=& \left\langle \r, \s, \t \right|\r^n=1,\s^2=1,\t^3=\r^{\frac{n}{2}},(\s\t)^4=1,\s\r\s^{-1}=\r^l,\t\r\t^{-1}=\r \rangle\\
G_{18}=& \left\langle \r, \s, \t \right|\r^n=1,\s^2=1,\t^3=1,(\s\t)^4=\r^{\frac{n}{2}},\s\r\s^{-1}=\r^l,\t\r\t^{-1}=\r \rangle\\
G_{19}=& \left\langle \r, \s, \t \right|\r^n=1,\s^2=1,\t^3=\r^{\frac{n}{2}},(\s\t)^4=\r^{\frac{n}{2}},\s\r\s^{-1}=\r^l,\t\r\t^{-1}=\r \rangle\\
G_{20}=& \left\langle \r, \s, \t \right|\r^n=1,\s^2=\r^{\frac{n}{2}},\t^3=1,(\s\t)^4=1,\s\r\s^{-1}=\r^l,\t\r\t^{-1}=\r \rangle \\
G_{21}=& \left\langle \r, \s, \t \right|\r^n=1,\s^2=\r^{\frac{n}{2}},\t^3=\r^{\frac{n}{2}},(\s\t)^4=1,\s\r\s^{-1}=\r^l,\t\r\t^{-1}=\r \rangle\\
\end{split}
\end{align*}
\begin{align*}
\begin{split}
G_{22}=& \left\langle \r, \s, \t \right|\r^n=1,\s^2=\r^{\frac{n}{2}},\t^3=1,(\s\t)^4=\r^{\frac{n}{2}},\s\r\s^{-1}=\r^l,\t\r\t^{-1}=\r \rangle\\
G_{23}=& \left\langle \r, \s, \t \right|\r^n=1,\s^2=\r^{\frac{n}{2}},\t^3=\r^{\frac{n}{2}},(\s\t)^4=\r^{\frac{n}{2}},\s\r\s^{-1}=\r^l,\t\r\t^{-1}=\r \rangle
\end{split}
\end{align*}
where $(l,n)=1$ and $l^2\equiv 1$ (mod n).
Furthermore $G_{16}\cong G_{17}$, $G_{18} \cong G_{19}$, $G_{20} \cong G_{21}$ and $G_{22} \cong G_{23}$.
\end{thm}

\begin{proof}
We know that the second cohomology group is as follows; see Table 1 in \cite{AK}.
\begin{equation*}
H^2(S_{4},C_{n}) \cong C_{(n,2)} \times C_{(n,2)}
\end{equation*}

(1) If $n$ is odd then $\left|H^2(S_{4},C_{n})\right|=1$. So $G \cong S_4 \times C_n$.

(2) If $n$ is even then $\left|H^2(S_{4},C_{n})\right|=4$. So there are at most 4 extensions which could occur. Let $C_n=\left\langle \r|\r^n=1\right\rangle$. The group $S_4$ has a presentation: $\left\langle \si, \ta ,\m| \si^2= \ta^3 = \m^4=1 \right\rangle$, where $\m=\si\ta$. Let $\s$ and $\t$ be the lifting of $\si$ and $\ta$ in $G$ respectively and we know that $C_n \triangleleft S_4$. Hence $\s\r\s^{-1}=\r^l$ and $\t\r\t^{-1}=r^k$. We choose $k=1$. By Remark ~\ref{r1}, $(l,n)=1$ and $l^2\equiv 1$ (mod $n$).

The case 16 in Table ~\ref{t2} is lifting of  $\si$, $\ta$ and $\m$ to elements of orders $|\si|$, $|\ta|$ and $|\m|$ respectively, then $\s^2=1,$ $\t^3=1$ and $(\s\t)^4=1$. Hence $G$ has a presentation as in $G_{16}$.

In case 17, $\ta$ lifts to an element of order $n|\ta|$, we choose $\t$ such that $\t^3=\r^{\frac{n}{2}}$. Since the orders of $\si$ and $\m$ are remaining the same, $\s^2=1$ and $(\s\t)^4=1$. Hence $G \cong G_{17}$.

If $\m$ lifts to an element of order $n|\m|$ like in case 18, then we choose $\s$ and $\t$ such that $(\s\t)^4=\r^{\frac{n}{2}}$. Since the orders of $\si$ and $\ta$ don't change, $\s^2=1$ and $\s\t^3=1$. Hence $G \cong G_{18}$.

If $\ta$ and $\m$ lift to elements of orders $n|\ta|$ and $n|\m|$ respectively, then we choose $\s$ and $\t$ such that $\t^3=\r^{\frac{n}{2}}$ and $(\s\t)^4=\r^{\frac{n}{2}}$. In case 19, we have such a situation. Since the order of $\si$ is remaining the same, $\s^2=1$. So $G$ has presentation as in $G_{19}$.

In case 20, $\si$ lifts to an element of order $n|\si|$, then we choose $\s$ such that $\s^2=\r^{\frac{n}{2}}$. Since the orders of $\ta$ and $\m$ are remaining the same, $\t^3=1$ and $(\s\t)^4=1$. So $G \cong G_{20}$.

If $\si$ and $\ta$ lift to elements of orders $n|\si|$ and $n|\ta|$ respectively like in case 21, then we choose $\s$ and $\t$ such that $\s^2=\r^{\frac{n}{2}}$ and $\t^3=\r^{\frac{n}{2}}$. Since the order of $\m$ does not change $(\s\t)^4=1$. Hence $G \cong G_{21}$.

In case 22, $\si$ and $\m$ lift to elements of orders $n|\si|$ and $n|\m|$ respectively, then we choose $\s$ and $\t$ such that $\s^2=\r^{\frac{n}{2}}$ and $(\s\t)^4=\r^{\frac{n}{2}}$. Since the order of $\m$ does not change $(\s\t)^4=1$. Hence $G$ has a presentation as in $G_{22}$.

In case 23, $\si$, $\ta$ and $\m$ lift to elements of orders $n|\si|$, $n|\ta|$ and $n|\m|$ respectively, then  we choose $\s$ and $\t$ such that $\s^2=\r^{\frac{n}{2}}$, $\t^3=\r^{\frac{n}{2}}$ and $(\s\t)^4=\r^{\frac{n}{2}}$. Hence $G \cong G_{23}$.

\end{proof}

\subsubsection{$\G \cong A_5$}

\begin{lemma} \label{l2}
Let $\G$ be either $A_{5}$ or $PSL_{2}(q)$. Then an extension of $\G$ by $C_{n}$ is central.
\end{lemma}

\begin{proof}
Let us consider the conjugation action of $A_{5}$ on $C_{n}$. The image of the induce homomorphism $\pi :A_{5} \longrightarrow Aut(C_{n})$ is a quotient of $A_{5}$. Since $Aut(C_{n})$ is abelian group, $im(\pi)$ is abelian group. $\G$ is non abelian simple group. Hence $\G$ is perfect group and $\frac{\G}{[\G,\G]}=1$. So $\G$ has only trivial abelian quotient. Therefore $im(\pi)$=1. Hence the action of $A_{5}$ on $C_{n}$ is trivial. Therefore extension of $A_{5}$ by $C_{n}$ is central.
\end{proof}

\begin{thm} \label{th5}
The automorphism groups for the cases 24-31 in Table ~\ref{t2} are as follows. If $n$ is odd or $G$ has a ramification as in cases 24-27 in Table ~\ref{t2} then $G \cong A_{5}\times C_{n}$. Otherwise $G$ admits group has presentation as:
\begin{center}
$\left\langle \r, \s, \t \right|\r^n=1,\s^2=\r^{\frac{n}{2}},\t^3=\r^{\frac{n}{2}},(\s\t)^5=\r^{\frac{n}{2}},\s\r\s^{-1}=\r,\t\r\t^{-1}=\r \rangle $
\end{center}
\end{thm}

\begin{proof}
By Lemma ~\ref{l2}, we know that extension is central and We know that the second cohomology group is as follows; see Table 1 in \cite{AK}.
\begin{equation*}
H^2(A_{5},C_{n}) \cong
\begin{cases}
1 & \text{if $n$ is odd}\\
C_{2} & \text{if $n$ is even}
\end{cases}
\end{equation*}
Hence if $n$ is odd there is only one central extension. Since $C_{n}$ is abelian $G \cong A_{5}\times C_{n}$. If $n$ is even there are two central extensions. The one possibility is $A_{5}\times C_{n}$. According to the ramification of the cases 28-31 in Table ~\ref{t2}, $G$ has element of order $2n$. But  $A_{5}\times C_{n}$ does not have element of order $2n$. Hence if $G$ has a ramification of the cases 24-27 in Table ~\ref{t2} then $G \cong A_{5}\times C_{n}$. Since $A_5=\left\langle  \s, \t \right|\s^2=\t^3=(\s\t)^5=1 \rangle$, all possible central extensions are of the form  $\left\langle \r, \s, \t \right|\r^n=1,\s^2=\r^{a},\t^3=\r^{b},(\s\t)^5=\r^{c},\s\r\s^{-1}=\r,\t\r\t^{-1}=\r \rangle$ where $a,b,c \in \{1,...,n\}$. If $a=b=c=1$ then the above presentation gives $A_{5}\times C_{n}$. We use GAP algebra package to calculate suitable $a$, $b$ and $c$ for the cases 24-27 in Table ~\ref{t2} and we found out that $a=b=c=\frac{n}{2}$.
\end{proof}

\subsubsection{$\G \cong U$} We defined $U=C_P^t$.

\begin{thm} \label{th12}
Let $\X_g$ be a genus $g\geq 2$ cyclic curve with $\G \cong U$. Then $G:=Aut(\X_g)$ as follows.

(1) If $G$ has ramification as in case 32 in Table ~\ref{t2} then $G$ has presentation:
\begin{multline*}
< \r,\s_1,\s_2,...,\s_t|\r^n=\s_1^p=\s_2^p=...=\s_t^p=1,\\
 \s_i\s_j=\s_j\s_i, \s_i\r\s_i^{-1}=\r^{l}, 1\leq i,j\leq t>
\end{multline*}
where $(l,n)=1$ and $l^p \equiv 1$ (mod n).

(2) If $G$ has ramification as in case 33 in Table ~\ref{t2} then $G \cong U \times C_n$.

\end{thm}

\begin{proof}
(1) Let $U= < \si_1,\si_2,...,\si_t|\si_1^p=\si_2^p=...=\si_t^p=1, \si_i\si_j=\si_j\si_i, 1\leq i,j\leq t>$. Let $C_n=\left\langle \r|\r^n=1\right\rangle$. Let $\s_1,\s_2,...,\s_t$ be the lifting of $\si_1,\si_2,...,\si_t$ in $G$ respectively. In case 32, $\si_1...\si_t$ lifts to an element of order $|\si_1...\si_t|$, then $\s_1^p=\s_2^p=...=\s_t^p=1, \s_i\s_j=\s_j\s_i, 1\leq i,j\leq t $. Since $C_n \triangleleft U$, $\s_i\r\s_i^{-1}=\r^{l_i}, 1\leq i\leq t $. By Remark ~\ref{r1}, $(l_i,n)=1$ and $l_i^p \equiv 1$ (mod n), for $1\leq i\leq t $. We choose $l=l_i$ for  $1\leq i\leq t $. Hence $G$ has presentation,
\begin{multline*}
< \r,\s_1,\s_2,...,\s_t|\r^n=\s_1^p=\s_2^p=...=\s_t^p=1, \s_i\s_j=\s_j\s_i, \s_i\r\s_i^{-1}=\r^{l}, 1\leq i,j\leq t>.
\end{multline*}

(2) In case 33, $G$ has an element of order $np^t$. We know that $(n,p)=1$ and $n|p^t-1$. Hence $(n,p^t)=1$. So among the extensions of $U$ by $C_n$, $U \times C_n$ is the only one extension for which has an element of order $np^t$. So in this case $G \cong U \times C_n$.

\end{proof}

\subsubsection{$\G \cong K_m$} We know that $K_m=C_p^t \rtimes C_m$ and $m|p^t-1$.

\begin{thm}\label{th13}
Let $\X_g$ be a genus $g\geq 2$ cyclic curve with $\G \cong K_m$. Then $G:=Aut(\X_g)$ as follows.

(1) If $G$ has ramification as in case 34 in Table ~\ref{t2} then $G$ has presentation:
\begin{multline*}
<\r,\s_1,...,\s_t,v|\r^n=\s_1^p=...=\s_t^p=v^m=1, \s_i\s_j=\s_j\s_i,\\
v\r v^{-1}=\r, \s_i\r\s_i^{-1}=\r^{l}, \s_iv\s_i^{-1}=v^{k}, 1\leq i,j\leq t >
\end{multline*}
where $(l,n)=1$ and $l^p \equiv 1$ (mod n), $(k,m)=1$ and $k^p \equiv 1$ (mod m).

(2) If $G$ has ramification as in case 35,36 and 37 in Table ~\ref{t2} then $G$ has presentation:
\begin{center}
$G_{35}=\left\langle \r,\s_1,...,\s_t|\r^{nm}=\s_1^p=...=\s_t^p=1, \s_i\s_j=\s_j\s_i, \s_i\r\s_i^{-1}=\r^{l}, 1\leq i,j\leq t\right\rangle $
\end{center}
where $(l,nm)=1$ and $l^p \equiv 1$ (mod nm).
\end{thm}

\begin{proof}
(1) Let $K=<\si_1,\si_2,...,\si_t,u|\si_1^p=\si_2^p=...=\si_t^p=u^m=1, \si_i\si_j=\si_j\si_i, \si_iu\si_i^{-1}=u^{k_i}, 1\leq i,j\leq t>$, $(k_i,m)=1$ and $k_i^p \equiv 1$ (mod m). Let $C_n=\left\langle \r|\r^n=1\right\rangle$. Let $\s_1,\s_2,...,\s_t,v$ be the lifting of $\si_1,\si_2,...,\si_t,u$ in $G$ respectively. In case 34, $u\si_1...\si_t$ lifts to an element of order $|u\si_1...\si_t|$, then $\s_1^p=\s_2^p=...=\s_t^p=v^m=1, \s_i\s_j=\s_j\s_i, \s v\s^{-1}=v^{k_i}, 1\leq i,j\leq t $. We choose $k=k_i$ for  $1\leq i\leq t $. Since $C_n \triangleleft U$, $\s_i\r\s_i^{-1}=\r^{l_i}, 1\leq i\leq t $.  By Remark ~\ref{r1}, $(l_i,n)=1$ and $l_i^p \equiv 1$ (mod n), for $1\leq i\leq t $. We choose $l=l_i$ for $1\leq i\leq t $. Also $v\r v^{-1}=\r^a$. By Remark ~\ref{r1}, $(a,n)=1$ and $a^m \equiv 1$ (mod n). we choose $a=1$. Hence $G$ has presentation,
\begin{multline*}
<\r,\s_1,...,\s_t,v|\r^n=\s_1^p=...=\s_t^p=v^m=1, \s_i\s_j=\s_j\s_i,\\
v\r v^{-1}=\r, \s_i\r\s_i^{-1}=\r^{l}, \s_iv\s_i^{-1}=v^{k}, 1\leq i,j\leq t >
\end{multline*}
(2) In case 35, 36 and 37, $G$ has elements of orders $nmp^t$ and $nm$. Among the extension of $K_m$ by $C_n$, $G_{35}$ is the only one extension so that it has elements of orders $nmp^t$ and $nm$. Non of other extensions have either elements of orders $nm$ or $nmp^t$. Note that if $(n,m)=1$, then $G_{35}$ is isomorphic to the group $G$ of case 34.
\end{proof}

\subsubsection{$\G \cong PSL_{2}(q)$}
We know that $q=p^f$ where $p$ is the characteristic of field $k$.

\begin{thm}\label{th6}
Let G be a Aut($\X_g$) where $\X_g$ is a cyclic curve of genus $g\geq2$ with $\G \cong PSL_{2}(q)$, $q\neq 9$ then G is as follows.

(1) If $G$ has ramification as in cases 38 and 39 then $G\cong PSL_2(q)\times C_n$.

(2) If $G$ has ramification as in cases 40 and 41 and $q=3$ then $G \cong SL_2(3)$. There are no possible groups for $q\neq 3$.
\end{thm}

\begin{proof}
By Lemma ~\ref{l2}, we know that extension is central and the second cohomology group is as follows; see Table 1 in \cite{AK}.
\begin{equation*}
H^2(PSL_2(q),C_{n}) \cong
\begin{cases}
1 & \text{if $p=2$, $p^f\neq4$}\\
C_{(2,n)} & \text{if $p>2$, $p^f\neq9$ or $p^f=4$}\\
C_{(6,n)} & \text{if $p^f=9$}
\end{cases}
\end{equation*}

(1) If $n$ is odd then there is only one extension. Since $C_n$ is abelian $G \cong PSL_2(q)\times C_n$. If $n$ is even, there are two extensions. According to ramification structure of cases 38 and 39 $G \cong PSL_2(q)\times C_n$. So for any $n$ $G \cong PSL_2(q)\times C_n$ for cases 38 and 39.

(2) By cases 40 and 41 of Theorem ~\ref{th1}, $n=2$. We know that $SL_2(q)$ is the only degree two central extension of $PSL_2(q)$; see \cite{IS}. If $q\neq 3$ then $SL_2(q)$ does not have elements of $n\alpha$ or $n\beta$. Therefore there are no possible groups for $q\neq 3$. But if $q=3$, $G \cong SL_2(3)$.

\end{proof}

\subsubsection{$\G \cong PGL_{2}(q)$} As previous subsection we know that $q=p^f$.

\begin{thm} \label{th11}
The automorphism group G such that $\G=G/C_n \cong PGL(2,q)$ is as follows. If $G$ has ramification as in cases 42 and 43 in Table ~\ref{t2} then $G \cong PGL(2,q) \times C_n$. There are no possible group extensions for the cases 44 and 45 in same table.

\end{thm}

\begin{proof}
We know that the second cohomology group is as follows; see Table 1 in \cite{AK}.
\begin{equation*}
H^2(PGL(2,q),C_{n}) \cong C_{(n,2)} \times C_{(n,2)}
\end{equation*}
According to the ramifications structure of the cases 42 and 43 and the second homology group, for any $n$, $G \cong PGL(2,q) \times C_n$. By Theorem ~\ref{th1}, only possible value for $n$ is one for the cases 44 and 45. Hence there are no possible groups for those cases.

\end{proof}

\subsection{The case p=5}
In this case $\G$ is isomorphic to one of the $C_m$, $D_m$, $A_4$, $S_4$, $U$, $K_m$, $PSL(2,q)$ or $PGL(2,q).$ Since the ramifications of covers $\phi:\bP \to \bP$ are similar to the ramifications in Theorem ~\ref{th1}, then signatures of covers $\Phi:\X \to \X^G$ and dimensions are same as corresponding cases in Table ~\ref{t2}.

\begin{thm}\label{th2}
Let $g \geq$ 2 be a fixed integer. Then the automorphism group $G$ of a cyclic curve of genus $g$ defined over a algebraically closed field $k$ such that char(k)=5 is one of the group in the Theorem ~\ref{th9}, ~\ref{th10}, ~\ref{th7}, ~\ref{th8}, ~\ref{th12}, ~\ref{th13},\\ ~\ref{th6}, ~\ref{th11}. Furthermore, signatures of covers $\Phi:\X \to \X^G$ and dimensions are same as corresponding cases in the Table ~\ref{t2}.
\end{thm}

\begin{proof}
Since the ramification of the  cases under $\G \cong C_m, D_m, A_4, S_4, U, K_m$, $PSL(2,q)$ and $PGL(2,q)$ are same as in Theorem ~\ref{th1}, the proofs of those cases are the same as proof of Theorem ~\ref{th1}. But the case $\G \cong A_5$ does not appear when $p=5$.
\end{proof}

\subsection{The case p=3}

In this case $\G \cong$ $C_m$, $D_m$, $A_5$, $U$, $K_m$, $PSL(2,q)$ or $PGL(2,q).$ The cases $\G \cong$ $C_m$, $D_m$, $U$, $K_m$, $PSL(2,q)$ and $PGL(2,q)$ have the same ramifications as in Theorem ~\ref{th1}. Hence those cases have signatures as in Table ~\ref{t2}. However the case $\G \cong$ $A_5$ has different ramification.

\begin{thm} \label{th3}
Let $g \geq$ 2 be a fixed integer. Then the automorphism group $G$ of a cyclic curve of genus $g$ defined over a algebraically closed field $k$ such that char(k)=3 is one of the group in the Theorems ~\ref{th9}, ~\ref{th10}, ~\ref{th12}, ~\ref{th13}, ~\ref{th6}, ~\ref{th11} or if $G$ has ramification as in cases $a,b$ in Table ~\ref{t3} then $G \cong A_5 \times C_n$. There are no possible group for cases $c,d$ in Table ~\ref{t3}.
Furthermore, signatures of covers $\Phi:\X \to \X^G$ and dimensions are same as corresponding cases in the Table ~\ref{t2} or Table ~\ref{t3}.
\end{thm}

\renewcommand{\arraystretch}{2}
\begin{longtable}{|c|c|c|c|}
\hline \hline
\multicolumn{1}{|c|}{$\#$} & \multicolumn{1}{|c|}{$\G$} & \multicolumn{1}{|c|}{$\delta(G,\C)$} & \multicolumn{1}{|c|}{$\C=(C_1,...,C_r)$}\\
\hline \hline
\endfirsthead
\hline \hline
\multicolumn{1}{|c|}{$\#$} & \multicolumn{1}{|c|}{$\G$} & \multicolumn{1}{|c|}{$\delta(G,\C)$} &  \multicolumn{1}{|c|}{$\C=(C_1,...,C_r)$}\\
\hline \hline
\endhead
\hline
\multicolumn{4}{r}{\itshape continued on the next page}\\
\endfoot
\multicolumn{2}{r}{ }
\endlastfoot
\hline
$a$ &  & $\frac{g+n-1}{30(n-1)}-1$ & $(6,5,n,...,n)$  \\
$b$ &  & $\frac{g+5n-5}{30(n-1)}-1$ & $(6,5n,n,...,n)$  \\
$c$ & $A_5$  & $\frac{g+6n-6}{30(n-1)}-1$ & $(6n,5,n,...,n)$ \\
$d$ & & $\frac{g}{30(n-1)}-1$ & $(6n,5n,n,...,n)$  \\
\hline \hline
\caption{The signature $\C$ and dimension $\d$ for $\G \cong A_5$ and $p=3$}\label{t3}
\end{longtable}

\begin{proof}
The proof of the  cases under $\G \cong C_m, D_m, A_4, S_4, U, K_m, PSL(2,q)$
and $PGL(2,q)$ are the same as in proof in Theorem ~\ref{th1}.

\textbf{Case $\G \cong A_5$ :} The ramification of $\p:\bP_x \to \bP_z$ is $(6,5)$. By Theorem 1 in \cite{VM}, the first point is wildly ramified and second one is tamely ramified. Hence in equation ~\ref{e2}, $\beta_1=e_1^*q_1+q_1-2$ for the first ramified place. By Theorem 1 in \cite{VM}, $q_1=3.$

(a) If both elements of orders 6 and 5 lift to elements of orders 5 and 6 then ramification is $\C=(6,5,n,...,n)$. In this case $e_1^*=2$. Hence by Riemann Hurwitz formula,
\[     2(g-1)=2(0-1)60n+60n\left(\left(\frac{6+3-2}{6}\right)+\left(\frac{5-1}{5}\right)+ \left(\frac{n-1}{n}\right)\left(\delta+1\right)\right).
\]
Then $\d=\frac{g+n-1}{30(n-1)}-1$.\\

(b) If element of order 5 lifts an element of order 5$n$, then ramification is $\C=(6,5n,n,...,n)$. As previous case $e_1^*=2$. Hence by Riemann Hurwitz formula,
\begin{multline*}
2(g-1)=2(0-1)60n+\\60n\left(\left(\frac{6+3-2}{6}\right)+\left(\frac{5n-1}{5n}\right)+ \left(\frac{n-1}{n}\right)\left(\delta+1\right)\right).
\end{multline*}
Then $\d=\frac{g-5n+5}{30(n-1)}-1$.\\

(c) If element of order 6 lifts an element of order 6$n$, then ramification is $\C=(6,5n,n,...,n)$. In this case $e_1^*=2n$. Furthermore $(2n,3)=1$ and $2n|(3-1)$. Hence only possible value for $n$ is one. So Riemann Hurwitz gives,
\begin{multline*}
2(g-1)=2(0-1)60n+\\60n\left(\left(\frac{6n+3-2}{6n}\right)+\left(\frac{5-1}{5}\right)+ \left(\frac{n-1}{n}\right)\left(\delta+1\right)\right).
\end{multline*}
Then $\d=\frac{g-6n+6}{30(n-1)}-1$.\\

(d) If both elements of orders 6 and 5 lift to elements of orders 5 and 6 then ramification is $\C=(6n,5n,n,...,n)$. As in previous case $e_1^*=2n$ and only possible value for $n$ is one.
\begin{multline*}
2(g-1)=2(0-1)60n+\\60n\left(\left(\frac{6n+3-2}{6n}\right)+\left(\frac{5n-1}{5n}\right)+ \left(\frac{n-1}{n}\right)\left(\delta+1\right)\right).
\end{multline*}
Then $\d=\frac{g}{30(n-1)}-1$.\\

By Lemma ~\ref{l2}, we know that extensions $A_5$ by $C_n$ is central. By Table 1 in  \cite{AK} $H^2(A_5,C_n)=C_{(n,2)}$. So if $G$ has ramification as in cases $a$ and $b$ then $G \cong A_5 \times C_n$. According to cases $c$ and $d$, only possible value for $n$ is one. So there are no possible group extensions for those two cases.
\end{proof}


\section{The Main Theorem}
We combine Theorems ~\ref{th9} - ~\ref{th3} altogether to make main theorem. This main theorem gives us all possible automorphism groups of genus $g\geq 2$ cyclic curves defined over the finite field of characteristic $p$.

\begin{thm} \label{th14}
Let $\X_g$ be a genus $g\geq2$ irreducible cyclic curve defined over an algebraically closed field $k$, $char(k)=p\neq2$, $G=Aut(\X_g)$, $\G$ its reduced automorphism group.

\begin{enumerate}

\item If $\G \cong C_m$ then $G \cong C_{mn}$ or
\begin{center}
$\left\langle \r, \s \right|\r^n=1,\s^m=1,\s\r\s^{-1}=\r^l \rangle$
\end{center}
where (l,n)=1 and $l^m\equiv 1$ (mod n).

\item If $\G \cong D_{2m}$ then $G \cong D_{2m} \times C_n$ or
\begin{align*}
\begin{split}
G_{5}=& \left\langle \r, \s, \t \right|\r^n=1,\s^2=\r,\t^2=1,(\s\t)^m=1,\s\r\s^{-1}=\r,\t\r\t^{-1}=\r^{n-1} \rangle\\
G_{6}=& D_{2mn} \\
G_{7}=& \left\langle \r, \s, \t \right|\r^n=1,\s^2=\r,\t^2=\r^{n-1},(\s\t)^m=1,\s\r\s^{-1}=\r,\t\r\t^{-1}=\r \rangle\\
G_{8}=& \left\langle \r, \s, \t \right|\r^n=1,\s^2=\r,\t^2=1,(\s\t)^m=\r^{\frac{n}{2}},\s\r\s^{-1}=\r,\t\r\t^{-1}=\r^{n-1} \rangle \\
G_{9}=& \left\langle \r, \s, \t \right|\r^n=1,\s^2=\r,\t^2=\r^{n-1},(\s\t)^m=\r^{\frac{n}{2}},\s\r\s^{-1}=\r,\t\r\t^{-1}=\r \rangle\end{split}
\end{align*}

\item If $\G \cong A_4$ and $p\neq3$ then $G \cong A_4 \times C_n$ or
\begin{align*}
\begin{split}
G'_{10}=& \left\langle \r, \s, \t \right|\r^n=1,\s^2=1,\t^3=1,(\s\t)^3=1,\s\r\s^{-1}=\r,\t\r\t^{-1}=\r^l \rangle\\
G'_{12}=& \left\langle \r, \s, \t \right|\r^n=1,\s^2=1,\t^3=\r^{\frac{n}{3}},(\s\t)^3=\r^{\frac{n}{3}},\s\r\s^{-1}=\r,\t\r\t^{-1}=\r^l \rangle\\
\end{split}
\end{align*}
where $(l,n)=1$ and $l^3\equiv 1$ (mod n) or
\begin{center}
$\left\langle \r, \s, \t \right|\r^n=1,\s^2=\r^{\frac{n}{2}},\t^3=\r^{\frac{n}{2}},(\s\t)^5=\r^{\frac{n}{2}},\s\r\s^{-1}=\r,\t\r\t^{-1}=\r \rangle $
\end{center}
or
\begin{align*}
\begin{split}
G_{10}=& \left\langle \r, \s, \t \right|\r^n=1,\s^2=1,\t^3=1,(\s\t)^3=1,\s\r\s^{-1}=\r,\t\r\t^{-1}=\r^k \rangle\\
G_{13}=& \left\langle \r, \s, \t \right|\r^n=1,\s^2=\r^{\frac{n}{2}},\t^3=1,(\s\t)^3=1,\s\r\s^{-1}=\r,\t\r\t^{-1}=\r^k \rangle\\
\end{split}
\end{align*}
where $(k,n)=1$ and $k^3\equiv 1$ (mod n).

\item If $\G \cong S_4$ and $p\neq3$ then $G \cong S_4 \times C_n$ or
\begin{align*}
\begin{split}
G_{16}=& \left\langle \r, \s, \t \right|\r^n=1,\s^2=1,\t^3=1,(\s\t)^4=1,\s\r\s^{-1}=\r^l,\t\r\t^{-1}=\r \rangle\\
G_{18}=& \left\langle \r, \s, \t \right|\r^n=1,\s^2=1,\t^3=1,(\s\t)^4=\r^{\frac{n}{2}},\s\r\s^{-1}=\r^l,\t\r\t^{-1}=\r \rangle\\
G_{20}=& \left\langle \r, \s, \t \right|\r^n=1,\s^2=\r^{\frac{n}{2}},\t^3=1,(\s\t)^4=1,\s\r\s^{-1}=\r^l,\t\r\t^{-1}=\r \rangle \\
G_{22}=& \left\langle \r, \s, \t \right|\r^n=1,\s^2=\r^{\frac{n}{2}},\t^3=1,(\s\t)^4=\r^{\frac{n}{2}},\s\r\s^{-1}=\r^l,\t\r\t^{-1}=\r \rangle\\
\end{split}
\end{align*}
where $(l,n)=1$ and $l^2\equiv 1$ (mod n).

\item If $\G \cong A_5$ and $p\neq5$ then $G \cong A_{5}\times C_{n}$ or
\begin{center}
$\left\langle \r, \s, \t \right|\r^n=1,\s^2=\r^{\frac{n}{2}},\t^3=\r^{\frac{n}{2}},(\s\t)^5=\r^{\frac{n}{2}},\s\r\s^{-1}=\r,\t\r\t^{-1}=\r \rangle $
\end{center}

\item If $\G \cong U$ then $G \cong U \times C_n$ or
\begin{multline*}
<\r,\s_1,\s_2,...,\s_t|\r^n=\s_1^p=\s_2^p=...=\s_t^p=1,\\ \s_i\s_j=\s_j\s_i, \s_i\r\s_i^{-1}=\r^{l}, 1\leq i,j\leq t>
\end{multline*}
where $(l,n)=1$ and $l^p \equiv 1$ (mod n).

\item If $\G \cong K_m$ then $G \cong$

\begin{multline*}
<\r,\s_1,...,\s_t,v|\r^n=\s_1^p=...=\s_t^p=v^m=1, \s_i\s_j=\s_j\s_i,\\
v\r v^{-1}=\r, \s_i\r\s_i^{-1}=\r^{l}, \s_iv\s_i^{-1}=v^{k}, 1\leq i,j\leq t >
\end{multline*}
where $(l,n)=1$ and $l^p \equiv 1$ (mod n), $(k,m)=1$ and $k^p \equiv 1$ (mod m) or
\begin{align*}
\begin{split}
\left\langle \r,\s_1,...,\s_t|\r^{nm}=\s_1^p=...=\s_t^p=1, \s_i\s_j=\s_j\s_i, \s_i\r\s_i^{-1}=\r^{l}, 1\leq i,j\leq t\right\rangle
\end{split}
\end{align*}
where $(l,nm)=1$ and $l^p \equiv 1$ (mod nm).

\item If $\G \cong PSL_{2}(q)$ then $G\cong PSL_2(q)\times C_n$ or $SL_2(3)$.

\item If $\G \cong PGL(2,q)$ then $G \cong PGL(2,q) \times C_n$.

\end{enumerate}
\end{thm}

\subsection{Automorphism groups of genus 3 cyclic curves}

Applying Theorem ~\ref{th1} through Theorem ~\ref{th3}, we obtain the automorphism groups of a genus 3 cyclic curve defined over algebraically closed field of characteristic 0,3,5,7 and bigger than 7. We listed GAP group ID of those groups in following theorem.

\begin{thm}
Let $\X_g$ be a genus 3 cyclic curve defined over a field of characteristic $p$. Then the automorphism groups of $\X_g$ are as follows.

\begin{description}
    \item[i)] $p=0$ : $(2,1)$, $(4,2)$, $(3,1)$, $(4,1)$, $(8,2)$, $(8,3)$, $(7,1)$, $(21,1)$, $(14,2)$, $(6,2)$, $(12,2)$,
    $(9,1)$, $(8,1)$, $(8,5)$, $(16,11)$, $(16,10)$, $(32,9)$, $(30,2)$, $(42,3)$, $(12,4)$, $(16,7)$, $(24,5)$, $(18,3)$,
    $(16,8)$, $(48,33)$, $(48,48)$.

    \item[ii)] $p=3$ : $(2,1)$, $(4,2)$, $(3,1)$, $(4,1)$, $(8,2)$, $(8,3)$, $(7,1)$, $(14,2)$, $(6,2)$, $(8,1)$, $(8,5)$, $(16,11)$, $(16,10)$, $(32,9)$, $(30,2)$, $(16,7)$, $(16,8)$, $(6,2)$.

    \item[iii)] $p=5$ : $(2,1)$, $(4,2)$, $(3,1)$, $(4,1)$, $(8,2)$, $(8,3)$, $(7,1)$, $(21,1)$, $(14,2)$, $(6,2)$, $(12,2)$,
    $(9,1)$, $(8,1)$, $(8,5)$, $(16,11)$, $(16,10)$, $(32,9)$, $(42,3)$, $(12,4)$, $(16,7)$, $(24,5)$, $(18,3)$,
    $(16,8)$, $(48,33)$, $(48,48)$.

    \item[iv)] $p=7$ : $(2,1)$, $(4,2)$, $(3,1)$, $(4,1)$, $(8,2)$, $(8,3)$, $(7,1)$, $(21,1)$, $(6,2)$, $(12,2)$,
    $(9,1)$, $(8,1)$, $(8,5)$, $(16,11)$, $(16,10)$, $(32,9)$, $(30,2)$, $(42,3)$, $(12,4)$, $(16,7)$, $(24,5)$, $(18,3)$,
    $(16,8)$, $(48,33)$, $(48,48)$.

    \item[v)] $p > 7$ : $(2,1)$, $(4,2)$, $(3,1)$, $(4,1)$, $(8,2)$, $(8,3)$, $(7,1)$, $(21,1)$, $(14,2)$, $(6,2)$, $(12,2)$,
    $(9,1)$, $(8,1)$, $(8,5)$, $(16,11)$, $(16,10)$, $(32,9)$, $(30,2)$, $(42,3)$, $(12,4)$, $(16,7)$, $(24,5)$, $(18,3)$,
    $(16,8)$, $(48,33)$, $(48,48)$.

\end{description}

\end{thm}

\subsection{Automorphism groups of genus 4 cyclic curves}

Again applying Theorem ~\ref{th1} through Theorem ~\ref{th3}, we obtain the following groups as automorphism groups of a genus 4 cyclic curve defined over algebraically closed field of characteristic 0,3,5,7 and bigger than 7. We listed GAP group ID of those groups in following theorem.

\begin{thm}
Let $\X_g$ be a genus 4 cyclic curve defined over a field of characteristic $p$. Then the automorphism groups of $\X_g$ are as follows.

\begin{description}
    \item[i)] $p=0$ : $(2,1)$, $(4,2)$, $(3,1)$, $(6,2)$, $(9,2)$, $(5,1)$, $(10,2)$, $(20,1)$, $(9,1)$, $(27,4)$, $(18,2)$, $(15,1)$, $(4,1)$, $(20,4)$, $(18,3)$, $(8,3)$, $(40,8)$, $(12,5)$, $(36,12)$, $(54,4)$, $(16,7)$, $(20,5)$, $(32,19)$, $(24,10)$, $(8,4)$, $(60,9)$, $(36,11)$, $(24,3)$, $(72,42)$.

    \item[ii)] $p=3$ : $(2,1)$, $(4,2)$, $(3,1)$, $(6,2)$, $(5,1)$, $(10,2)$, $(20,1)$, $(9,1)$, $(18,2)$, $(15,1)$, $(4,1)$, $(20,4)$, $(8,3)$, $(40,8)$, $(12,5)$, $(16,7)$, $(20,5)$, $(32,19)$, $(24,10)$, $(8,4)$, $(9,2)$, $(18,5)$.

    \item[iii)] $p=5$ : $(2,1)$, $(4,2)$, $(3,1)$, $(6,2)$, $(9,2)$, $(5,1)$, $(10,2)$, $(20,1)$, $(9,1)$, $(27,4)$, $(18,2)$, $(4,1)$, $(18,3)$, $(8,3)$, $(12,5)$, $(36,12)$, $(54,4)$, $(16,7)$, $(20,5)$, $(32,19)$, $(24,10)$, $(8,4)$, $(60,9)$, $(36,11)$, $(24,3)$, $(72,42)$, $(10,2)$, $(18,5)$.

    \item[iv)] $p=7$ : $(2,1)$, $(4,2)$, $(3,1)$, $(6,2)$, $(9,2)$, $(5,1)$, $(10,2)$, $(20,1)$, $(9,1)$, $(27,4)$, $(18,2)$, $(15,1)$, $(4,1)$, $(20,4)$, $(18,3)$, $(8,3)$, $(40,8)$, $(12,5)$, $(36,12)$, $(54,4)$, $(16,7)$, $(20,5)$, $(32,19)$, $(24,10)$, $(8,4)$, $(60,9)$, $(36,11)$, $(24,3)$, $(72,42)$.

    \item[v)] $p > 7$ : $(2,1)$, $(4,2)$, $(3,1)$, $(6,2)$, $(9,2)$, $(5,1)$, $(10,2)$, $(20,1)$, $(9,1)$, $(27,4)$, $(18,2)$, $(15,1)$, $(4,1)$, $(20,4)$, $(18,3)$, $(8,3)$, $(40,8)$, $(12,5)$, $(36,12)$, $(54,4)$, $(16,7)$, $(20,5)$, $(32,19)$, $(24,10)$, $(8,4)$, $(60,9)$, $(36,11)$, $(24,3)$, $(72,42)$.

\end{description}

\end{thm}

\end{document}